\documentclass[11pt]{amsart}

\usepackage{amsmath}
\usepackage{amssymb}
\usepackage{mathrsfs}
\usepackage[all]{xy}

\newtheorem{theorem}{Theorem}[section]
\newtheorem{claim}[theorem]{Claim}

\newtheorem{lemma}[theorem]{Lemma}

\newtheorem{proposition}[theorem]{Proposition}
\newtheorem{corollary}[theorem]{Corollary}

\theoremstyle{definition}
\newtheorem{definition}[theorem]{Definition}

\newtheorem{question}[theorem]{Question}

\theoremstyle{remark}
\newtheorem{remark}[theorem]{Remark}

\def\l{{\langle}}
\def\r{{\rangle}}

\newcount\skewfactor
\def\mathunderaccent#1#2 {\let\theaccent#1\skewfactor#2
\mathpalette\putaccentunder}
\def\putaccentunder#1#2{\oalign{$#1#2$\crcr\hidewidth
\vbox to.2ex{\hbox{$#1\skew\skewfactor\theaccent{}$}\vss}\hidewidth}}

\def\smallbox#1{\leavevmode\thinspace\hbox{\vrule\vtop{\vbox
   {\hrule\kern1pt\hbox{\vphantom{\tt/}\thinspace{\tt#1}\thinspace}}
   \kern1pt\hrule}\vrule}\thinspace}

\DeclareMathOperator{\id}{id}

\DeclareMathOperator{\crit}{crit}
\DeclareMathOperator{\Sp}{Sp}
\DeclareMathOperator{\dpt}{dp}
\DeclareMathOperator{\Cub}{Cub}

\newcommand{\cf}{{\rm cf}}
\def\Ult{{\rm Ult}}

\title{Measures that violate the generalized continuum hypothesis}
\author{Tom Benhamou}
\address[Benhamou]{Department of Mathematics, Rutgers University, Piscataway NJ 08854-
8019, USA.}
\email{tom.benhamou@rutgers.edu}
\thanks{The research of the first author was supported by the National Science Foundation under Grant
No. DMS-2346680}
\author{Gabriel Goldberg}
\address[Goldberg]{Departement of Mathematics, UC Berkeley, Berkeley, CA 94720-3840 USA }
\thanks{The research of the second author was supported by the National Science Foundation under Grant
No. DMS-2401789}
\email{ggoldberg@berkeley.edu}

\begin{document}
\begin{abstract}
    A simple \(P_\lambda\)-point on a regular cardinal \(\kappa\) is a uniform ultrafilter on \(\kappa\) with a mod-bounded decreasing generating sequence of length \(\lambda\). 
    We prove that if there is a simple $P_\lambda$-point ultrafilter over $\kappa>\omega$, then $\lambda=\mathfrak{d}_\kappa=\mathfrak{b}_\kappa=\mathfrak{u}_\kappa=\mathfrak{r}_\kappa=\mathfrak{s}_\kappa$. We show that such ultrafilters appear in the models of \cite{SimonOmer,BROOKETAYLOR201737}. We improve the lower bound for the consistency strength of the existence of a $P_{\kappa^{++}}$-point to a $2$-strong cardinal. Finally, we apply our arguments to obtain non-trivial lower bounds for (1) the statement that the generalized tower number $\mathfrak{t}_\kappa$ is greater than $\kappa^+$ and $\kappa$ is measurable, (2) the preservation of measurability after the generalized Mathias forcing, and (3) variations of filter games of \cite{NIELSEN_WELCH_2019,HolySchlicht:HierarchyRamseyLikeCardinals,MagForZem} in the case $2^\kappa>\kappa^+$. 
\end{abstract}
\maketitle
\section{introduction}
For a cardinal $\lambda$, a point $x$ in a topological space $X$ is called a $P_\lambda$-point if the intersection of fewer than $\lambda$-many open neighborhoods of $x$ contains an open neighborhood of $x$. Of course, every isolated point is a $P_\lambda$-point for every $\lambda$. Interpreting this definition in the space $U(Y)$ of uniform ultrafilters on \(Y\) gives rise to the notion of a $P_\lambda$-point ultrafilter, which translates to the following combinatorial condition: a uniform ultrafilter $U$ over $Y$ is a $P_\lambda$-point if the poset $(U,\supseteq^*)$ is $\lambda$-directed.\footnote{For $A,B\subseteq Y$, the relation $A\subseteq^* B$ stands for $|A\setminus B|<|Y|$.} Namely, for any $\mu<\lambda$ and any collection $\l X_i\mid i<\mu\r\subseteq U$ there is $X\in U$ such that $X\subseteq^* X_i$ for all $i<\mu$. This type of ultrafilter on $\omega$ has been studied in numerous papers (e.g. \cite{BlassMildenberger,BlassShelah,BrendleShelah,Nyikos2020-NYISUA}). On regular uncountable cardinals, relatively little is known. Baker and Kunen \cite{BakerKunen} have some constructions of such ultrafilters and lately the first author \cite{tomCohesive} used such ultrafilters to address a question of Kanamori regarding cohesive ultrafilters from \cite{Kanamori1978}.

The notion of a $P_\lambda$-point ultrafilter has appeared naturally in classical constructions. The most relevant one here is due to Kunen \cite[Chapter VIII Ex. (A10)]{kunen}, which used a finite support iteration of the Mathias forcing (see  \ref{Def:Mathias forcing}) to construct an ultrafilter on \(\omega\) which is generated by fewer than $\mathfrak{c}$-many sets. The Mathias forcing associated to an ultrafilter \(U\in \beta(\omega)\setminus \omega\) is a ccc forcing that adds a subset of $\omega$ that is eventually included in every set in $U$. By iterating Mathias forcings associated to a carefully chosen sequence of ultrafilters, Kunen adds a $\subseteq^*$-decreasing sequence of sets, and by performing an iteration whose length \(\lambda\) has uncountable cofinality, he produces a sequence that generates an ultrafilter in the generic extension. This ultrafilter is a $P_{\cf(\lambda)}$-point which is moreover simple: a \textit{simple} $P_\mu$-point is an ultrafilter $U$ that has a generating sequence $\l X_i\mid i<\mu\r\subseteq U$ that is $\subseteq^*$-decreasing. 

In an unpublished work, Carlson generalized Kunen's construction to construct a simple $P_\lambda$-point on a measurable cardinal, starting from a supercompact cardinal. This establishes the consistency of a $\kappa$-complete ultrafilter over a measurable cardinal $\kappa$ which is generated by fewer than $2^\kappa$-many sets. The question of the consistency strength of a uniform ultrafilter on a measurable cardinal $\kappa$ which is generated by fewer than $2^\kappa$-many sets remains open.

\subsection*{Cardinal characteristics at measurable cardinals} Unlike the situation on countable sets, the generalized Kunen method is currently the only known method to separate the generalized ultrafilter number $\mathfrak{u}_\kappa$ from the powerset of a measurable cardinal and therefore plays an important role in the landscape of the recent interest in generalized cardinal characteristics \cite{SimonOmer, OmerMoti,BROOKETAYLOR201737,FreidmanThompson,LambieHanson,ZapletalSplitting}.

There are several known techniques for controlling generalized cardinal invariants \cite{FischerGaps,CUMMINGSShelah,BrendleFreidmanMontoya}, all of which are incompatible with controlling the ultrafilter number. Brook-Taylor, Fischer, Friedman, and Montoya \cite{BROOKETAYLOR201737} used variations of the generalized Kunen construction to establish that it is consistent for many generalized cardinal characteristics to be equal yet smaller than $2^\kappa$. Their forcing adds a simple \(P_\lambda\)-point. In Section \S2, we show that the existence of a simple \(P_\lambda\)-point alone implies the equality of many of these characteristics.\footnote{
Let us mention that in the model of \cite{BROOKETAYLOR201737}, there are other characteristics, such as $\mathfrak{i}_\kappa,\mathfrak{p}_\kappa,\mathfrak{a}_\kappa,\mathfrak{t}_\kappa$, and various invariants of category, that also coincide with the value of $\lambda$. We do not address these cardinals in this paper. 
}
More precisely, we prove the following theorem:
\begin{theorem}\label{Main Thm 1}
    Suppose \(\kappa < \lambda\) are regular uncountable cardinals and there is a simple $P_\lambda$-point on \(\kappa\). Then $$\mathfrak{u}_\kappa=\mathfrak{u}^{com}_\kappa=\mathfrak{b}_\kappa=\mathfrak{d}_\kappa=\mathfrak{s}_\kappa=\mathfrak{r}_\kappa = \lambda.$$
    In particular, if \(\mu\neq \lambda\) is regular, then there are no simple \(P_\mu\)-points on \(\kappa\).
\end{theorem}

The effect of a simple $P_\lambda$-point on cardinal characteristics on $\omega$ was already noticed by Nyikos \cite{Nyikos2020-NYISUA} and further investigated by Blass and Shelah \cite{BlassShelah}, and Brendle--Shelah \cite{BrendleShelah}. Nyikos proved that if there is a simple $P_\lambda$-point on $\omega$, then either $\lambda=\mathfrak{b}_\kappa$ or $\mathfrak{d}_\kappa$. In sharp contrast to Theorem \ref{Main Thm 1}, Br\"{a}uninger--Mildenberger \cite{MildenBraun} recently showed that it is consistent for there to be a simple $P_\lambda$-point and a simple $P_\mu$-point for $\mu\neq\lambda$. 

Theorem \ref{Main Thm 1} shows that new methods are needed to obtain a model with a small ultrafilter number $\mathfrak{u}_\kappa$ which is not, for example, equal to the bounding number $\mathfrak{b}_\kappa$ or the dominating number $\mathfrak{d}_\kappa$. (Of course, one can add many Cohen functions to $\kappa$, which blows up $\mathfrak{u}_\kappa$ to \(2^\kappa\) while preserving $\mathfrak{b}_\kappa$.)
\begin{question}
    Is it consistent with a measurable cardinal to have $\mathfrak{d}_\kappa<\mathfrak{u}_\kappa<2^\kappa$? how about $\mathfrak{b}_\kappa<\mathfrak{u}_\kappa<2^\kappa$?
\end{question}

Another method for dealing with cardinal characteristics at the level of a measurable cardinal is the extender-based Magidor--Radin forcing of Merimovich \cite{CarmiRadin}. In particular, Ben-Neria--Gitik \cite{OmerMoti} and Ben-Neria--Garti \cite{SimonOmer} used this technique to obtain results regarding the splitting number $\mathfrak{s}_\kappa$ and reaping number $\mathfrak{r}_\kappa$ at this level. To generalize the above analysis of cardinal characteristics to this framework, we introduce the notion of a \textit{simple pseudo-\(P_\lambda\)-point} (see Definition \ref{def: pi Point})
and show:
\begin{theorem}
    In the model of \cite{SimonOmer}, there is a simple pseudo-$P_{\kappa^+}$-point.
\end{theorem}

\begin{theorem}\label{Main Thm 2}
    If there is a  simple pseudo-$P_\lambda$-point, then
    $$\lambda=\pi\mathfrak{u}_\kappa=\mathfrak{b}_\kappa=\mathfrak{d}_\kappa=\mathfrak{s}_\kappa=\mathfrak{r}_\kappa.$$
\end{theorem} 

We also reduce the large cardinal upper bound of the claim ``$\kappa$ is measurable and $\mathfrak{r}_\kappa<2^\kappa$" below $o(\kappa)=\kappa^{+3}$. 

\subsection*{The consistency strength of a $P_{\kappa^{++}}$-point.}
\cite{tomCohesive} raises the question: what is the consistency strength of the existence of a $P_\lambda$-point for $\lambda>\kappa^+$? As we mentioned, it is possible to start with an indestructible supercompact cardinal and force such an ultrafilter, but this is clearly an overkill since a supercompact cardinal cannot be the first $\alpha$ such that $\alpha$ carries a $P_{\alpha^{++}}$-point. A trivial lower bound comes from the fact that we have to blow up the powerset of a measurable cardinal for such an ultrafilter to exist, and by Mitchell--Gitik \cite{Mithandbook}, this implies an inner model with a measurable cardinal $\kappa$ of Mitchell order $o(\kappa)=\kappa^{++}$. Gitik proved \cite[Thm. 5.2]{tomCohesive} that $o(\kappa)=\kappa^{++}$ is not enough and at least an inner model with a  $\mu$-measurable cardinal is required.\footnote{A $\mu$-measurable cardinal is cardinal $\kappa$ which is the critical point of an elementary embedding $j:V\to M$ such that $\{X\subseteq \kappa\mid \kappa\in j(X)\}\in M$. Such a cardinal is a limit of cardinals $\delta$ with $o(\delta)=2^{2^\delta}$.} Here we improve this lower bound to a $2$-strong cardinal, and more generally:
\begin{theorem}\label{Main them 4}
Suppose that the core model ${K}$ exists, and that in $V$ there is a measurable cardinal $\kappa$ carrying a $P_\lambda$-point for some $\lambda>\kappa^+$ regular. Then there is an inner model with a $\lambda$-strong cardinal.
\end{theorem}
The proof uses an analysis of the iterated ultrapower of $K$ arising from the restriction of $j_U$ to \(K\), where $U$ is a $P_\lambda$-point. 

Finally, we provide three applications of  this type of lower bound. The first is to show that the statement that $\mathfrak{t}_\kappa > \kappa^+$, where \(\mathfrak{t}_\kappa\) is the generalized tower number associated to a measurable cardinal \(\kappa\),
has consistency strength greater than $o(\kappa)=\kappa^{++}$. This is related to the result of  Zapletal \cite{ZapletalSplitting} and Ben-Neria--Gitik \cite{OmerMoti} that the statement ``$\mathfrak{s}_\kappa>\kappa^{+}$ for a regular $\kappa$" is equiconsistent with $o(\kappa)=\kappa^{++}$. Since $\mathfrak{t}_\kappa\leq \mathfrak{s}_\kappa$, then $\mathfrak{t}_\kappa>\kappa^{+}$ for a regular cardinal $\kappa$ is also at least at the level of $o(\kappa)=\kappa^{++}$. The following improves this when adding the measurability of $\kappa$:
\begin{theorem}
    Suppose that $\kappa$ is measurable and $\mathfrak{t}_\kappa>\kappa^+$ then there is an inner model with a $\mu$-measurable.
\end{theorem}

The second application is to show that the generalization of Kunen's construction cannot be carried from the  assumption of $o(\kappa)=\kappa^{++}$:
\begin{corollary}
    Let $\kappa$ be measurable in $V$, and $U
    \in V$ be a $\kappa$-complete ultrafilter over $\kappa$. Suppose that $V\subseteq M$ is a larger model in which $\kappa$ is measurable and $M$ contains and  $V$-generic set for the generalized Mathias forcing $\mathbb{M}_U$. Then in ${K}$ there is a $\mu$-measurable cardinal.
\end{corollary}
Hence if one wishes to obtain a small ultrafilter number at a measurable cardinal from optimal assumptions, then a new method is required.

The third application relates to the \textit{filter games} of Holy-Schlicht \cite{HolySchlicht:HierarchyRamseyLikeCardinals}, Nielsen-Welch \cite{NIELSEN_WELCH_2019} and Foreman-Magidor-Zeman \cite{MagForZem}. These games revolve around the following idea: two players, Player I and Player II take turns. First, Player I plays a submodel \(M\) of \(H(\kappa^+)\) of size $\kappa$ and Player II responds with an object that determines a $\kappa$-complete (or even normal) ultrafilter on that model. In one variant of the game, the object played by Player II is an \(M\)-ultrafilter, but in another variant, Player II is required to play a single set, external to \(M\), that generates an \(M\)-ultrafilter modulo bounded subsets of \(\kappa\). In the next round, Player I extends \(M\) to a model \(M'\) and Player II must extend the previous ultrafilter to measure sets in \(M'\). 

Under the assumption of $2^\kappa=\kappa^+$, the existence of a winning strategy for Player II (in either of the games) is equivalent to $\kappa$ being measurable. Here, we consider these games of length $\gamma$, where $\gamma\in [\kappa^+,2^\kappa)$. Our main observation is that the consistency strength of a winning strategy for Player II in the game where they play filters is still just a measurable cardinal, and that the consistency strength jumps past $o(\kappa)=\kappa^{++}$ (again, involving $\mu$-measures) if Player II is required to play sets. 

This paper is organized as follows:
\begin{itemize}
    \item In Section~\ref{Crushing Cardinals}, we present our results regarding cardinal characteristics and simple $P_\lambda$-points. In Subsection \ref{SubSection: pi-characters} we focus on the $\pi$-character variations and in  Subsection \ref{Subsection: ExtenderModel} we consider the Extender-based Magidor-Radin model.
    \item In Section~\ref{Section :Lower bounds}, we provide our lower bound on the existence of a $P_\lambda$-point.
    \item In Sections~\ref{Section: Applications},\ref{Sec: Math},\ref{Sec: Filter} we prove our three applications.
\end{itemize}
\subsection*{Notation}\label{Sec: Notations}
For a set $X$ and a cardinal $\alpha$ we let $[X]^\alpha=\{Y\subseteq X\mid |Y|=\alpha\}$. For $A\in [\kappa]^\kappa$ we let $f_A:\kappa\to\kappa$ be the increasing enumeration of the set $A$. Namely, $f_A$ is the inverse of the transitive collapse of $A$. Given two ultrafilters $U,W$ on $X,Y$ resp. we say that $U\leq_{RK} W$ if there is a function $f:Y\to X$ such that $A\in U$ iff $f^{-1}[A]\in W$. A \textit{measurable cardinal} is an uncountable cardinal $\kappa$ such that there is a non-trivial $\kappa$-complete ultrafilter on $\kappa$. Given an ultrafilter $U$ over $X$, we let $j_U:V\to Ult(V,U)\simeq M_U$ be the usual ultrapower embedding associated to an ultrafilter, and $M_U$ is the Mostowski collapse of $Ult(V,U)$ (which we identify with $M_U$ from this point on, whenever $Ult(V,U)$ is well-founded).
A \textit{$\lambda$-supercompact cardinal} is a cardinal $\kappa$ such that there is a $\kappa$-complete fine normal ultrafilter on $P_\kappa(\lambda)$. A \textit{supercompact cardinal} is a $\lambda$-supercompact for every $\lambda$. A \textit{$\lambda$-strong cardinal} is a cardinal $\kappa$ such that there is an elementary embedding $j:V\to M$ with $crit(j)=\kappa$, $M$ is closed under $\kappa$-sequences and $V_{\kappa+\lambda}\subseteq M$. A \textit{$\mu$-measurable cardinal} was defined in footnote 3. This is equivalent to the existence of a  \textit{$\mu$-measure}; that is, a $\kappa$-complete ultrafilter $U$ over $\kappa$ such that $\{X\subseteq\kappa\mid \kappa\in j_U(X)\}\in M_{U}$.\footnote{ For the non-trivial direction, fix $j:V\to M$ as in footnote 3. Since \(|V_\kappa| = \kappa\), it suffices to find an ultrafilter \(U\) over \(V_\kappa\) such that \(\{X\subseteq \kappa : \kappa\in j_U(X)\}\in M_U\). To do this, let $U$ be the ultrafilter over $V_\kappa$ derived from the point $D=\{X\subseteq\kappa\mid \kappa\in j(X)\}\in M$, noting that \(D\in j(V_\kappa)\). 
Let \(k: M_U\to M\) be the factor embedding given by \(k([f]_U) = j(f)(D)\). Let \(\bar D = [\text{id}]_U\).
Then \(k(\bar D) = D\).
In particular, \(\kappa = \bigcup D = k(\bigcup \bar D)\in \text{ran}(k)\), and it follows that \(\crit(k) > \kappa\). Therefore \(k\) is the identity on \(P(\kappa)\), and so the fact that \(k(\bar D) = D\) implies \(\bar D = k^{-1}[D] = D\). Finally, if \(X\subseteq \kappa\), we have \(\kappa\in j_U(X)\) if and only if \(\kappa = k(\kappa) \in k(j_U(X)) = j(X)\); so \(\{X\subseteq \kappa :\kappa\in j_U(X)\} = D\in M_U\), as desired.}

If \(M\) is a transitive model
of \(\text{ZFC}^-\) and \(X\in M\),
an \textit{\(M\)-ultrafilter} on \(X\)
is an ultrafilter \(U\) on the Boolean algebra \(P(X)\cap M\).
The ultrapower of \(M\)
by \(U\) is the quotient
of \(M^X\cap M\) under the  equivalence relation associated to \(U\). An \textit{\(M\)-normal ultrafilter} (also known as an \(M\)-normal \(M\)-ultrafilter)
is an \(M\)-ultrafilter \(U\) on an ordinal \(\kappa\in M\) such that for any sequence \(\langle A_\alpha \rangle_{\alpha < \kappa}\in M\), \(\Delta_{\alpha < \kappa} A_\alpha\in U\).

If \(\kappa\) is a cardinal and \(U\) is an \(M\)-ultrafilter, \(U\) is \textit{\(\kappa\)-complete}
if \(U\) extends to a \(\kappa\)-complete filter in \(V\), or equivalently, if the intersection of fewer than \(\kappa\) elements of \(U\) is nonempty. In general, this is distinct from the notion of \textit{\(M\)-\(\kappa\)-completeness},
which only requires that if
\(\bigcap_{\alpha < \eta} A_\alpha\in U\) whenever \(\langle A_\alpha \rangle_{\alpha < \eta}\in U^{<\kappa} \cap M\).
Since the models \(M\) we consider are usually closed under sequences of length less than \(\kappa\), this distinction will not be important here.
\section{Crushing cardinal characteristics}\label{Crushing Cardinals}
Let $\kappa$ be a regular uncountable cardinal. We denote by ${}^\kappa\kappa$ the set of all functions $f:\kappa\to \kappa$. On ${}^\kappa\kappa$ we have the almost everywhere domination order denoted by $\leq^*$, and defined by  $$f\leq^* g\text{  iff }\exists \alpha<\kappa\ \forall \alpha\leq \beta<\kappa, \ f(\beta)\leq g(\beta).$$
\begin{definition}The generalized bounding and dominating numbers are defined as follows:
    \begin{enumerate}
        \item $\mathfrak{b}_\kappa=\min\{|\mathcal{A}|\mid \mathcal{A}\subseteq \kappa^\kappa\text{ is unbounded in }({}^{\kappa}\kappa,\leq^*)\}$.
        \item $\mathfrak{d}_\kappa=\min\{|\mathcal{A}|\mid \mathcal{A}\subseteq \kappa^\kappa\text{ is dominating in }({}^{\kappa}\kappa,\leq^*)\}$.
    \end{enumerate}
\end{definition}
These cardinal invariants can be characterized using the club filter $$\Cub_\kappa=\{A\subseteq \kappa\mid \exists C \text{ closed unbounded in }\kappa, \ C\subseteq A\}.$$ The \textit{almost inclusion order} denoted by $\subseteq^*$ is defined by $A\subseteq^* B$ iff 
$\exists \alpha<\kappa, A\setminus \alpha\subseteq B.$

\begin{proposition}[Folklore]\label{prop:club char of b,d} \ {}
    \begin{enumerate}
        \item $\mathfrak{b}_\kappa=\min\{|\mathcal{A}|\mid \mathcal{A}\subseteq \Cub_\kappa\text{ is unbounded in }(\Cub_\kappa,\supseteq^*)\}$.
        \item $\mathfrak{d}_\kappa=\min\{|\mathcal{A}|\mid \mathcal{A}\subseteq \Cub_\kappa\text{ is cofinal in }(\Cub_\kappa,\supseteq^*)\}$.
    \end{enumerate}
\end{proposition}
\begin{proof}
For $(2)$, see \cite[Claim 4.8]{bgp}. For $(1)$, 
let us first prove that \(\mathfrak{b}_\kappa\) 
is bounded above by the size of any unbounded subset of \((\Cub_\kappa,\supseteq^*)\).
Let $\mathcal{A}\subseteq \Cub_\kappa$ we claim that the set $\{f_A\mid A\in\mathcal{A}\}$ of increasing enumerations of sets in \(\mathcal {A}\) is unbounded in $({}^\kappa\kappa,\leq^*)$. Otherwise, let $f$ be a $\leq^*$ bound and let $C_f$ be the club of closure points of $f$. We claim that $C_f\subseteq^* A$ for all $A\in\mathcal{A}$. Indeed, let $\alpha$ be such that for every $\alpha\leq \beta<\kappa$, $f_A(\beta)\leq f(\beta)$. If $\gamma\in C_f\setminus \alpha$, then for every $\beta\in\gamma\setminus \alpha$, $\beta\leq f_A(\beta)\leq f(\beta)<\gamma$. Since \(f_A(\beta)\in A\), it follows that $\gamma$ is a limit point of $A$. Since $A$ is a club, $\gamma\in A$. This proves \(C_f\setminus \alpha\subseteq A\), as desired. 

For the opposite inequality, suppose that $\mathcal{S}$ is unbounded in $({}^\kappa\kappa,\leq^*)$.  Let $\{C_f\mid f\in\mathcal{S}\}$ be the collection of clubs of closure points of elements of $\mathcal{S}$. We claim that $\{C_f\mid f\in\mathcal{S}\}$ is unbounded. Otherwise, suppose that $C\subseteq^* C_f$ for all $f\in \mathcal{S}$. Define $g(\alpha)=f_C(\alpha+1)$. We claim that $g$ dominates $\mathcal{S}$, which would lead to a contradiction. To see this, let $\alpha$ be such that $f_{C}(\alpha)=\alpha=f_{C_f}(\alpha)$ and $C\setminus \alpha+1\subseteq C_f\setminus \alpha+1$.
This implies that for \(\beta \geq \alpha\), \(f_{C_f}(\beta) \leq f_C(\beta)\). 
Therefore given $\beta>\alpha$, notice that $\beta<f_{C_f}(\beta+1)\in C_f$, hence \[f(\beta)<f_{C_f}(\beta+1)\leq f_C(\beta+1)=g(\beta)\qedhere\]
\end{proof}
Given an ultrafilter $U$ on a cardinal $\kappa\geq\omega$, let $j_U:(V,\in)\to (M_U,\in_U)$ be the usual ultrapower construction. Then $(j_U(\kappa),\in_U)=({}^\kappa\kappa/U,<_U)$ is a linear order and $\cf^V(j_U(\kappa))$ is a regular cardinal.
\begin{claim}
    For every uniform ultrafilter $U$ over $\kappa$, $\mathfrak{b}_\kappa\leq \cf^V(j_U(\kappa))\leq \mathfrak{d}_\kappa$.
\end{claim}
\begin{proof}
    Clearly, if $\mathcal{A}$ is dominating in $({}^\kappa\kappa,\leq^*)$, then $\{[f]_U\mid f\in\mathcal{A}\}$ is cofinal in $j_U(\kappa)$. On the other hand if $\{[f_\alpha]_U\mid \alpha<\lambda\}$ is cofinal in $j_U(\kappa)$. Then it must be unbounded in $({}^\kappa\kappa,\leq^*)$, since if $g:\kappa\to\kappa$ was a bound in $\leq^*$, then $[g]_U<j_U(\kappa)$ would bound $\{[f_\alpha]_U\mid \alpha<\lambda\}$, which is supposed to be cofinal.
\end{proof}

Given a filter $F$ on $\kappa$ we say that $\mathcal{B}$ is a base for $F$ if $\mathcal{B}\subseteq F$ and for every $A\in F$, there is $B\in \mathcal{B}$ such that $B\subseteq^* A$. Define:
\begin{enumerate}
    \item  $\mathfrak{ch}(F)=\min\{|\mathcal{B}|\mid \mathcal{B}\text{ is a base for }F\}$ is the character of $F$.
\item  $\mathfrak{u}_\kappa=\min\{\mathfrak{ch}(U)\mid U\text{ is a uniform ultrafilter on }\kappa\}$ is the ultrafilter number
\item 
 $\mathfrak{u}^{com}_\kappa=\min\{\mathfrak{ch}(U)\mid U\text{ is a }\kappa\text{-complete ultrafilter on }\kappa\}$ is the complete ultrafilter number
\end{enumerate}
The depth spectrum, introduced in \cite{tomCohesive}, is the set
$\Sp_{\dpt}(F)$ of all regular cardinals $\lambda$ for which there exists a $\subseteq^*$-decreasing sequence \(\l X_i\mid i<\lambda\r\subseteq F\) with no \(\subseteq^*\)-lower bound in \(F\).
Also define the \textit{depth of }$F$ by:
$$\mathfrak{t}(F)=\min \Sp_{\dpt}(F)$$
\begin{remark}
    Note that $\mathfrak{t}(F)$ is a regular cardinal. The notation emphasizes that $\mathfrak{t}(F)$ is an analog of the well-known tower number. In \cite[Prop. 4.14]{tomCohesive} it was shown that $\mathfrak{t}(F)=\min(\Sp_T(F,\supseteq^*))$ where $\Sp_T(F,\supseteq^*)=\{\lambda\in Reg\mid \lambda\leq_T (F,\supseteq^*)\}$. Here $\leq_T$ is the well-known Tukey order (see for example \cite{DobrinenTukeySurvey15}).
\end{remark}
In the case $F=\Cub_\kappa$, it is not hard to see that $\mathfrak{t}(\Cub_\kappa)=\mathfrak{b}_\kappa$ and $\mathfrak{ch}(\Cub_\kappa)=\mathfrak{d}_\kappa$.
 \begin{claim}
    Let $U,W$ be ultrafilters. If $U\leq_{RK}W$ then
    $$\mathfrak{t}(W)\leq \mathfrak{t}(U), \ \mathfrak{ch}(U)\leq \mathfrak{ch}(W).$$
\end{claim}
\begin{proof}
    The right inequality is well-known, and the left follows from the fact that if $U\leq_{RK} W$ implies that $(U,\supseteq^*)\leq_{T} (W,\supseteq^*)$ (see for example \cite[Fact 1]{Dobrinen/Todorcevic11}) and therefore $\Sp_T(U,\supseteq^*)\subseteq \Sp_T(W,\supseteq^*)$ which ultimately implies $\mathfrak{t}(W)\leq \mathfrak{t}(U)$.
\end{proof}
\begin{proposition}\label{Prop: d less ch, dp less b} Let $U$ be a $\kappa$-complete ultrafilter over $\kappa$. Then:
\begin{enumerate}
    \item 
     $\mathfrak{d}_\kappa\leq \mathfrak{ch}(U)$.
    \item 
    $\mathfrak{t}(U)\leq \mathfrak{b}_\kappa$ \end{enumerate}
\end{proposition}
    \begin{proof}
        For $(1)$, let $U^*$ be a normal ultrafilter RK-below $U$, then $\mathfrak{ch}(U^*)\leq \mathfrak{ch}(U)$. Let $\mathcal{B}$ be a base for $U^*$ and  $\mathcal{C}=\{\overline{b}\mid b\in\mathcal{B}\}\subseteq \Cub_\kappa$. We claim that $\mathcal{C}$ is a generating set for $\Cub_\kappa$. Given any club $C$, since $U^*$ is normal, $C\in U^*$ and therefore there is $b\in\mathcal{B}$ such that $b\subseteq^* C$. Since $C$ is closed, $\overline{b}\subseteq^* C$, as wanted.

    For $(2)$, again we may assume that $U$ is normal. Note that every sequence of clubs $\l C_i\mid i<\kappa\r$ for $\kappa<\mathfrak{t}(U)$ has a lower bound in $U$ and therefore the closure of that lower bound would be a club-bound. Hence $\mathfrak{t}(U)\leq \mathfrak{b}_\kappa$. 
\end{proof}
\begin{lemma}\label{Lemma: b less than u less that u com}
    $\mathfrak{b}_\kappa\leq\mathfrak{u}_\kappa\leq \mathfrak{u}^{com}_\kappa$
\end{lemma}
\begin{proof}
    The nontrivial inequality $\mathfrak{b}_\kappa\leq\mathfrak{u}_\kappa$ will follow from a more general fact regarding the reaping number in Lemma \ref{Lemma: pi-base is unbounded} and Theorem \ref{Thm: Raghavan-Shelah}.
\end{proof}
\begin{definition}
For a uniform filter $F$ over $\kappa$, we say that:
\begin{enumerate}
    \item $F$ is a $P_\lambda$-point if $(F,\supseteq^*)$ is $\lambda$-directed. Namely, if for every $\mathcal{A}\subseteq F$, $|\mathcal{A}|<\lambda$, there is $B\in F$ such that $B\subseteq^* A$ for all $A\in\mathcal{A}$.
    \item $F$ is a simple $P_\lambda$-point if there is a $\subseteq^*$-decreasing sequence $\l X_i\mid i<\lambda\r\subseteq F$ that forms a base for $F$.
    \item $\mathfrak{p}(F)=\min\{\lambda\mid F\text{ is not a }P_{\lambda^+}\text{-point}\}$.
\end{enumerate}  
\end{definition}
Note that $F$ is a simple $P_\lambda$-point if and only if $F$ is a simple $P_{\cf(\lambda)}$-point. Hence we will only consider simple $P_\lambda$-point for regular $\lambda$'s. Also, note that if $U$ is a uniform ultrafilter that is a simple $P_\lambda$-point over \(\kappa\), then $\lambda$ must be at least $\kappa^+$, and therefore $U$ must be $\kappa$-complete. It was proven in \cite{tomCohesive} that $\mathfrak{t}(F)=\mathfrak{p}(F)$. In \cite[Lemma 4.23]{tomCohesive} it was proven that
    $F$ is a simple $P_\lambda$-point if and only if $\mathfrak{t}(F)=\mathfrak{ch}(F)=\lambda$.
\begin{corollary}\label{cor:cub-b-d}
    For a regular  cardinal $\lambda$, $\Cub_\kappa$ is a simple $P_\lambda$-point if and only if $\lambda=\mathfrak{d}_\kappa=\mathfrak{b}_\kappa$.
\end{corollary}
\begin{theorem}
    If \(\kappa <\lambda\) are regular uncountable cardinals and $U$ is a simple $P_\lambda$-point ultrafilter on \(\kappa\), then $\lambda=\mathfrak{d}_\kappa=\mathfrak{b}_\kappa=\mathfrak{u}_\kappa=\mathfrak{u}^{com}_\kappa$.
\end{theorem}
\begin{proof}
    Indeed, by Proposition \ref{Prop: d less ch, dp less b}(2), and Lemma \ref{Lemma: b less than u less that u com}, $\lambda=\mathfrak{t}(U)\leq \mathfrak{b}_\kappa\leq \mathfrak{u}_\kappa$. Also, by \ref{Prop: d less ch, dp less b}(1) $\mathfrak{d}_\kappa\leq \mathfrak{ch}(U)=\lambda$ and clearly $\mathfrak{u}_\kappa\leq \mathfrak{u}^{com}_\kappa\leq \mathfrak{ch}(U)=\lambda$. So by the fact that $U$ is a simple $P_\lambda$-point we get the desired equality. 
\end{proof}
    \begin{corollary}
        If $\mu$ and $\lambda$ are regular and there are simple $P_\lambda$-point and $P_\mu$-point ultrafilters over $\kappa>\omega$, then $\mu=\lambda$.
    \end{corollary}
    This is not the case on $\omega$. Nyikos \cite{Nyikos2020-NYISUA} showed that
    the set of regular cardinals \(\lambda\) for which there is a simple \(P_\lambda\)-point ultrafilter on \(\omega\) has cardinality at most two; recently,
    Br\"{a}uninger--Mildenberger \cite{MildenBraun} proved a spectacular result that it is consistent with ZFC that there are simple $P_{\aleph_1}$-point and  $P_{\aleph_2}$-point ultrafilters on \(\omega\).
    \begin{corollary}
        For a regular uncountable cardinal $\kappa$, if there is a simple $P_\lambda$-point ultrafilter over $\kappa$, then $\cf(j_U(\kappa))=\lambda$ for every uniform ultrafilter on $\kappa$.
    \end{corollary}
\subsection{$\pi$-characters, splitting, and reaping numbers}\label{SubSection: pi-characters}
Let us consider a well-known weakening of the characteristics from the previous section. 
We say that $\mathcal{B}$ is a $\pi$-base for a uniform ultrafilter $U$ on $\kappa$ if $\mathcal{B}\subseteq [\kappa]^\kappa$ and for every $A\in U$, there is $B\in \mathcal{B}$ such that $B\subseteq^* A$. 
$$\pi\mathfrak{ch}(U)=\min\{|\mathcal{B}|\mid \mathcal{B}\text{ is a }\pi\text{-base for }U\}$$
$$\pi\mathfrak{u}_\kappa=\min\{\pi\mathfrak{ch}(U)\mid U\text{ is a uniform ultrafilter over }\kappa\}$$
$$\pi\mathfrak{u}^{com}_\kappa=\min\{\pi\mathfrak{ch}(U)\mid U\text{ is a }\kappa\text{-complete ultrafilter over }\kappa\}$$
Clearly, the above characteristics are all less than or equal to their respective $\pi$-free versions.
The $\pi$-depth spectrum is the set $\Sp_{\pi \dpt}(U)$ of regular cardinals $\lambda$ for which there exists a \(\subseteq^*\)-decreasing sequence $\l X_i\mid i<\lambda\r\subseteq U$ that is unbounded in $([\kappa]^\kappa,\supseteq^*)$. From this we can define the \(\pi\)-analog of \(\mathfrak{t}\):
$$\pi\mathfrak{t}(U)=\min \Sp_{\pi \dpt}(U)$$
\begin{definition}\label{def:pi_plambda_point}
       $U$ is a $\pi P_\lambda$-point if every $\mathcal{A}\subseteq U$ of cardinality less than $\lambda$ has a pseudo-intersection. Namely there is $B\in [\kappa]^\kappa$ such that $B\subseteq^* A$ for all $A\in\mathcal{A}$.
\end{definition}
Once again, we note that we may restrict our attention to  $\pi P_\lambda$-points where $\lambda$ is regular and that such a lambda must be of cofinality at least $\kappa^+$.

$$\pi\mathfrak{p}(U)=\min\{\lambda\mid U\text{ is not a }\pi P_{\lambda^+}\text{-point}\}$$
\begin{remark}
    \begin{enumerate}
        \item $\Sp_{\pi \dpt}(U)\subseteq \Sp_{\dpt}(U)$.
\item   $\mathfrak{t}(U)\leq \pi\mathfrak{p}(U)\leq \pi\mathfrak{t}(U)\leq \pi\mathfrak{ch}(U)\leq \mathfrak{ch}(U)$. The inequalities $\pi\mathfrak{p}(U)\leq \pi\mathfrak{t}(U)$  and $\pi\mathfrak{ch}(U)\leq \mathfrak{ch}(U)$ are immediate from the definitions. To see $\pi\mathfrak{t}(U)\leq \pi\mathfrak{ch}(U)$ suppose towards a contradiction that $\pi\mathfrak{t}(U)=\lambda_1> \pi\mathfrak{ch}(U)=\lambda_0$, let $\l X_i\mid i<\lambda_1\r\subseteq U$ be $\subseteq^*$-decreasing witnessing $\lambda_1\in \Sp_{\pi \dpt}(U)$, and let $\l b_\alpha\mid \alpha<\lambda_0\r$ be a $\pi$-base for $U$. For each $X_i$, there is some ${\alpha_i} < \lambda_0$ such that $b_{\alpha_i}\subseteq^* X_i$. There are unboundedly many $i$'s such that $\alpha_i=\alpha^*$ and therefore $b_{\alpha^*}$ would be a lower bound for $\l X_i\mid i<\lambda_1\r$ in $([\kappa]^\kappa,\subseteq^*)$, contradiction.

For $\mathfrak{t}(U)\leq \pi\mathfrak{p}(U)$, recall that $\mathfrak{t}(U)=\mathfrak{p}(U)$ and if $U$ is not a $\pi P_{\lambda^+}$-point then $U$ is also not a $\pi P_{\lambda^+}$-point.
    \end{enumerate}
\end{remark}
\begin{question}
    Is $\pi\mathfrak{t}(U)=\pi\mathfrak{p}(U)$?
\end{question}
\begin{remark}\label{Club remark}
    One can define the above $\pi$-characteristics for filters. For the club filter however, we have that $\pi\mathfrak{ch}(\Cub_\kappa)=\mathfrak{ch}(\Cub_\kappa)$, $\pi\mathfrak{t}(\Cub_\kappa)=\mathfrak{t}(\Cub_\kappa)$, and $\pi\mathfrak{p}(\Cub_\kappa)=\mathfrak{p}(\Cub_\kappa)$.
\end{remark}
We say that $f:\kappa\to \kappa$ is almost one-to-one modulo an ultrafilter $U$ if there is $X\in U$ such that $f\restriction X$ is bounded-to-one, namely, for every $\gamma<\kappa$, $\pi^{-1}[\{\gamma\}]\cap X$ is bounded in $\kappa$. The following is a generalization of the well known Rudin-Blass ordering of ultrafilters on $\omega$:
 \begin{definition}
    Let $U,W$ be ultrafilters over $\kappa$. We say that an ultrafilter $U$ is Rudin-Blass below $W$, and denote it by $U\leq_{RB}W$ if there is an almost one-to-one mod $W$ function $f:\kappa\to\kappa$ such that $f_*(W)=U$.
\end{definition}
\begin{theorem}[Kanamori, Ketonen]\label{Thm: moving to ext of Club}
    Let $U$ be a  countably complete uniform ultrafilter over a regular cardinal $\kappa$. Then $U$ is $RB$-above an ultrafilter which extends the club filter.
\end{theorem}
\begin{proof}
    First, we claim that if \(W\) is an uniform ultrafilter on a regular uncountable cardinal $\kappa$ such that no function that is almost one-to-one modulo $W$ is regressive on a  set in $W$, then $W$ extends the club filter. To see this, note that any nonstationary set $A\subseteq \kappa$
    supports a monotone regressive function $g : A\to \kappa$.
    (Namely, let $C\subseteq \kappa \setminus A$ be club, and let $g(\alpha) = \sup(C\cap \alpha)$ for $\alpha\in A$.)
    Therefore $W$ cannot contain a nonstationary set, and hence $W$ extends the club filter.

    To prove the theorem, let $f : \kappa\to \kappa$ be the \(<_U\)-least function that is almost one-to-one modulo \(U\), and let \(W = f_*(U)\). Note that \(W\leq_{RB} U\) is a uniform ultrafilter on $\kappa$ such that no function that is almost one-to-one modulo $W$ is regressive on a  set in $W$, and hence $W$ extends the club filter.
\end{proof}
\begin{remark}\label{remark: adapts to models}
    The assumption of countable completeness in the previous theorem can be improved to the assumption that there is a least almost one-to-one function modulo \(U\). Also, the argument adapts to countably complete $M$-ultrafilters where $M$ is a transitive model of $\text{ZFC}^{-}$.
\end{remark}
\begin{theorem}\label{Thm: Rudin-Blass inequalities}
    If $U\leq_{RB}W$ then $\pi \mathfrak{t}(W)\leq \pi\mathfrak{t}(U)$ and $\pi\mathfrak{ch}(U)\leq \pi\mathfrak{ch}(W)$. 
\end{theorem}
\begin{proof}
    Let $g:\kappa\to\kappa$ be such that $g_*(W)=U$ and let $X\in W$ be such that $g\restriction X$ is almost one-to-one. Let $\l X_i\mid i<\lambda\r$ be a $\pi$-base for $W$. By shrinking the sequence to another $\pi$-base, we may assume that for every $i<\lambda$, $X_i\subseteq^*X$. This means that $g[X_i]$ must be unbounded in $\kappa$. It is clear now that $\l g[X_i]\mid i<\lambda\r$ is a $\pi$-base for $U$. 
    For the other inequality, let $\l Y_i\mid i<\lambda\r\subseteq U$ be $\subseteq^*$-decreasing with no pseudo-intersection. Then $\l g^{-1}[Y_i]\mid i<\lambda\r$ must also be $\subseteq^*$-decreasing. If $Y$ would have been a pseudo-intersection, then $g[Y]$ would have been a pseudo-intersection of the $Y_i$'s. Note that if we start with a sequence $\l Z_i\mid i<\lambda\r\subseteq W$ with no pseudo-intersection, then $g[Z_i]$ is indeed $\subseteq^*$-decreasing, but this sequence might have a pseudo-intersection.  
\end{proof}
\begin{theorem}\label{Thm:b=d=lambda}
    For any countably complete uniform ultrafilter $U$ on $\kappa$, $\pi\mathfrak{ch}(U)\geq \mathfrak{d}_\kappa$ and $\pi\mathfrak{t}(U)\leq \mathfrak{b}_\kappa$.
\end{theorem}
\begin{proof}
    By Theorem \ref{Thm: moving to ext of Club}, we can find $U^*\leq_{RB} U$ such that $U^*$ extends the club filter.  By Theorem \ref{Thm: Rudin-Blass inequalities} it sufficed to prove the inequalities for $U^*$. The argument for $U^*$ is a straightforward generalization of Proposition \ref{Prop: d less ch, dp less b}.
\end{proof}
The countably completeness assumption will be removed using Lemma \ref{Lemma: pi-base is unbounded} and Theorem \ref{Thm: Raghavan-Shelah}.

Let us introduce the splitting and reaping numbers.
    We say that $A$ splits $B$ if $A\cap B$ and $B\setminus A$ are unbounded in $\kappa$.
    We say that $\mathcal{A}$ is a splitting family if every $X\in [\kappa]^\kappa$ is splittable by some $A\in\mathcal{A}$. We say that $\mathcal{A}\subseteq [\kappa]^\kappa$ is unsplittable, if there is no $A\in [\kappa]^\kappa$ that splits every $A\in \mathcal{A}$.\begin{enumerate}
        \item $\mathfrak{s}_\kappa=\min\{|\mathcal{A}|\mid \mathcal{A}\text{ is a splitting family}\}$.
        \item $\mathfrak{r}_\kappa=\min\{|\mathcal{A}|\mid \mathcal{A}\text{ is a 
 unsplittable family}\}$.
    \end{enumerate} 
 Evidently,  $\mathcal{A}$ is  unsplittable if for example, $\mathcal{A}$ is a $\pi$-base of a uniform ultrafilter. Hence $\mathfrak{r}_\kappa\leq \pi\mathfrak{ch}(U)$. In fact B. Balcar and P. Simon proved that $\mathfrak{r}_\kappa$ is always realized by a $\pi$-base of a uniform ultrafilter \cite{BALCAR1991133}.
 \begin{lemma}\label{Lemma: pi-base is unbounded} Let $U$ be a uniform ultrafilter over $\kappa$.
    \begin{enumerate}
        \item $\pi\mathfrak{ch}(U)\geq\mathfrak{r}_\kappa$.
        \item $\pi\mathfrak{p}(U)\leq \mathfrak{s}_\kappa$
    \end{enumerate}
\end{lemma}
\begin{proof}
$(1)$ is trivial as we observed above. For $(2)$, let $\l S_j\mid j<\mathfrak{s}_\kappa\r$ be a splitting family. For every $j$, either $S_j$ or $\kappa\setminus S_j$ is in $U$. If  $\mathfrak{s}_\kappa<\pi\mathfrak{p}(U)$, these sets would have had a pseudo-intersection which couldn't be split by any of the $S_j$'s. This is a contradiction.
\end{proof}
   
\begin{theorem}[Raghavan-Shelah {\cite{RaghavanShelah}}]\label{Thm: Raghavan-Shelah} 
    Let $\kappa$ be an inaccessible cardinal, then: \begin{enumerate}
        \item $\mathfrak{d}_\kappa\leq \mathfrak{r}_\kappa$
    \item $\mathfrak{s}_\kappa\leq \mathfrak{b}_\kappa$.
    \end{enumerate}
\end{theorem}
The following is a generalization of a simple $P_\lambda$-point.
\begin{definition}\label{def: pi Point}
   We say that an ultrafilter $U$ is a simple $\pi P_\lambda$-point if $\pi\mathfrak{p}(U)=\lambda=\pi\mathfrak{ch}(U)$
\end{definition}
Since $\mathfrak{t}(U)\leq\pi\mathfrak{p}(U)\leq \pi\mathfrak{ch}(U)\leq \mathfrak{ch}(U)$, a simple $P_\lambda$-point is a simple $\pi P_\lambda$-point.
\begin{corollary}
    If there is a uniform simple $\pi P_\lambda$-point on \(\kappa\) then $\lambda=\pi\mathfrak{u}_\kappa=\mathfrak{d}_\kappa=\mathfrak{b}_\kappa=\mathfrak{s}_\kappa=\mathfrak{r}_\kappa$.
\end{corollary}
\begin{proof}
    This follows from  Theorem \ref{Thm:b=d=lambda}, Lemma \ref{Lemma: pi-base is unbounded}, Theorem \ref{Thm: Raghavan-Shelah}.
\end{proof}
\begin{question}
    What about $\mathfrak{a}_\kappa,\mathfrak{i}_\kappa,\mathfrak{p}_\kappa,\mathfrak{t}_\kappa$? Are they determined in the presence of a simple $P_\lambda$-point?
\end{question}
\subsection{Another look at the extender-based model}\label{Subsection: ExtenderModel}
In \cite{OmerMoti}, Ben-Neria and Gitik used the Merimovich extender-based Magidor-Radin forcing from \cite{CarmiRadin} in order to prove that it is consistent that the splitting number at a regular uncountable cardinal $\kappa$ is a regular cardinal $\lambda >\kappa^+$ from the existence of a measurable $\kappa$ with $o(\kappa)=\lambda$.

The following summarizes the relevant properties of a generic extension $M=V[G]$ via the extender based Magidor-Radin forcing: $\kappa<\lambda$ are regular uncountable cardinals of $M$ and there are intermediate models $\l M_i\mid i<\lambda\r$ of ZFC and sequences $\l\mathcal{U}_i\mid i<\lambda\r$ and $\l k_i\mid i<\lambda\r$ in \(M\) such that: 
\begin{enumerate}
        \item If $i<j$ then $M_i\subseteq M_j$.
        \item $\mathcal{U}_i\in M_i$ and $M_i\models \mathcal{U}_i$ is a normal ultrafilter .
        \item $k_i\in[\kappa]^\kappa$ diagonalizes $\mathcal{U}_i$ (i.e. $k_i\subseteq^* X$ for every $X\in \mathcal{U}_i$). Also $k_j\in M_{i}$ for all $j<i$. 
        \item $P(\kappa)^M=\bigcup_{i<\lambda}P(\kappa)^{M_i}$.
    \end{enumerate}
 In \cite{tomCohesive}, these properties were used to prove that in $V[G]$, the club filter is a simple $P_\lambda$-point. Combining this with \ref{cor:cub-b-d}:
 \begin{corollary}
     If $(1)$ through $(4)$ hold, then $M\models\mathfrak{b}_
     \kappa=\mathfrak{d}_\kappa=\lambda$.
 \end{corollary}

Let us show how to deduce that the splitting number is large: 
 
 \begin{proposition}\label{Prop: sk in radin}
    If  $(1)$ through $(4)$ hold, then $M\models\mathfrak{s}_\kappa= \lambda$.
\end{proposition}
\begin{proof}
    Since $\mathfrak{s}_\kappa\leq\mathfrak{d}_\kappa$, it suffices to prove that $\lambda\leq\mathfrak{s}_\kappa$. Suppose that $\mathcal{S}\in M$ is a collection of subsets of $\kappa$ of size less than $\lambda$. By items $(1)$ and $(4)$, there is some $i<\lambda$ such that $\mathcal{S}\subseteq P(\kappa)^{M_i}$. By $(2)$, for each $X\in \mathcal{S}$, either $X\in \mathcal{U}_i$ or $\kappa\setminus X\in U_i$. By $(3)$, $k_i$ diagonalizes $\mathcal{U}_i$, and therefore, for each $X\in \mathcal{S}$,  either $k_i\subseteq^* X$ or $k_i\subseteq^*\kappa\setminus X$. So $k_i$ is not split by any member of $\mathcal{S}$.  
\end{proof}
The conditions $(1)$ through $(4)$ also determine the value of the reaping number:
\begin{proposition}\label{Prop: rk in radin}
    If $(1)$ through $(4)$ hold, then $M\models\mathfrak{r}_\kappa=\cf(\lambda)$.
\end{proposition}
\begin{proof}
    Again, since $\mathfrak{b}_\kappa\leq \mathfrak{r}_\kappa$, it suffices to prove that $\mathfrak{r}_\kappa\leq \cf(\lambda)$. Let $\{\alpha_i\mid i<\cf(\lambda)\}\in M$ be cofinal in $\lambda$. We claim that $\{ k_{\alpha_i}\mid i<\cf(\lambda)\}$ is a reaping family. To see this, let $X\in M$ be any subset of $\kappa$. By $(4)$ there is $i$ such that $X\in M_i$. Let $i_0<\lambda$ such that $i\leq \alpha_{i_0}$. By $(1)$, $X\in M_{\alpha_{i_0}}$ and by $(2)$, either $X\in \mathcal{U}_{\alpha_{i_0}}$ or $\kappa\setminus X\in \mathcal{U}_{\alpha_{i_0}}$. By $(3)$, $k_{\alpha_{i_0}}\subseteq^* X$ or $k_{\alpha_{i_0}}\subseteq^* \kappa\setminus X$, as desired. 
\end{proof}
\begin{corollary}
    In the models of \cite{OmerMoti}, $\mathfrak{b}_\kappa=\mathfrak{d}_\kappa=\mathfrak{r}_\kappa=\mathfrak{s}_\kappa=\kappa^{++}=2^\kappa$.
\end{corollary}
\begin{corollary}
    In the models of \cite{SimonOmer}, $\mathfrak{b}_\kappa=\mathfrak{d}_\kappa=\mathfrak{r}_\kappa=\mathfrak{s}_\kappa=\kappa^{+}<2^\kappa$.
\end{corollary}
This reduces the upper bound on the consistency results obtained by Brooke-Taylor--Fischer--Friedman--Montoya \cite{BROOKETAYLOR201737} from a supercompact cardinal to the low levels of strong cardinals.

To obtain the configuration of the reaping number above, Ben-Neria and Garti \cite{SimonOmer} prove that some of the ultrafilters $\mathcal{U}_i$ cohere, that is:
\begin{itemize}
    \item [(5)] There is an unbounded $S\subseteq \lambda$ such that for every $i<j$ in $S$, $\mathcal{U}_i\subseteq \mathcal{U}_j$.
\end{itemize}
They used (5), for example, to deduce that $\kappa$ is measurable in $M$. In fact, the ultrafilter they produce is a $\kappa$-complete simple $\pi P_{\lambda}$-point:
\begin{theorem}
    Assume that $\l \mathcal{U}_i,\mid i<\lambda\r,\l k_i\mid i<\lambda\r\in M$ and $(1)$ through $(5)$ hold. Then in $M$ there is a normal ultrafilter $\mathcal{U}$ which is a simple $\pi P_{\lambda}$-point. In particular, $\pi\mathfrak{u}^{com}_\kappa=\lambda$.
\end{theorem}
\begin{proof}
    Consider the ultrafilter $\mathcal{U}=\bigcup_{i\in S}\mathcal{U}_i$. It is easy to see that $\pi \mathfrak{ch}(\mathcal{U})\leq \lambda$. We claim that $\lambda\leq\pi \mathfrak{p}(U)$, which finishes the proof. Suppose that $\l X_i\mid i<\rho\r\subseteq U$, for some $\rho<\lambda$. Then, similar arguments show that there is $j<\lambda$ such that $k_j$ is a pseudo-intersection for the sequence  $\l X_i\mid i<\rho\r$. 
\end{proof}
It is an open problem whether one can obtain \(\mathfrak u_\kappa = \kappa^+ < 2^\kappa\) at an inaccessible cardinal \(\kappa\) from much less than a supercompact cardinal. The previous theorem shows that current techniques suffice to obtain the analogous result for \(\pi\mathfrak{u}_\kappa\) from hypotheses at the level of strong cardinals.

In fact, to obtain a model $M$ satisfying $(1)$ through $(5)$, the authors of \cite{SimonOmer} used a measurable cardinal $\kappa$ such that $o(\kappa)$ is a weakly compact cardinal above \(\kappa\). However, if we only wish to keep $\kappa$ measurable and play with the values of $\mathfrak{r}_\kappa$ and $\mathfrak{s}_\kappa$, we only need to secure (1) through (4), and therefore we can get away with much less; for example, \(o(\kappa) = \kappa^{+4}\) suffices. (This
uses \cite[Claim 5.9]{CarmiRadin}
to ensure the preservation of measurability.)

\begin{question}
    Can one determine the values of other generalized cardinal characteristics at $\kappa$ in the extender-based Magidor-Radin model?
\end{question}
\section{Lower Bounds}\label{Section :Lower bounds}
\subsection{The Strength of a \(P_\lambda\)-point}
Gitik showed that if there is $P_{\kappa^{++}}$-point then there is an inner model with a $\mu$-measurable. The argument can be found in \cite{tomCohesive}. In terms of consistency strength, this is already above $o(\kappa)=\kappa^{++}$. Here we improve his result a bit.

\begin{lemma}\label{lma:factor_crit}
    Suppose \(j :  V\to M\) is an elementary embedding with critical point \(\kappa\) and \(\alpha < ((2^\kappa)^+)^M\). Let \(D\) be the ultrafilter on \(\kappa\) derived from \(j\) using \(\alpha\) and \(k : M_D\to M\) be the canonical factor embedding. Then \(\crit(k) > \alpha.\)
    \begin{proof}
        Let \(f\in \text{ran}(k)\) be a surjection from \(P(\kappa)\) onto \(\alpha+1\), which exists since \(k[M_D]\) is an elementary substructure of \(M\) and \(\{\kappa,\alpha\}\in k[M_D]\).
        Since \(P(\kappa)\subseteq M_D\), we have \(P(\kappa)\subseteq \text{ran}(k)\). Hence \(\alpha + 1 = f[P(\kappa)]\subseteq \text{ran}(k)\).
    \end{proof}
\end{lemma}
\begin{theorem}\label{Thm: p-pointtostrong}
    If there is a $P_{\kappa^{++}}$-point $U$, then there is an inner model with a \(2\)-strong cardinal.
\end{theorem}
\begin{proof}
    Assume towards a contradiction that there is no inner model with a \(2\)-strong cardinal. Let $E_0$ be the first extender used in the unique normal iteration \(i : K\to j_U(K)\). Note that this iteration exists and \(i = j_U\restriction K\) by Schindler's theorem \cite[Corollary 3.1]{schindler_2006}. (In fact, for core models at the level of strong cardinals, the theorem is due to Steel \cite[Theorem 8.13]{SteelBook}.)
    Then $j_U\restriction {K}=k\circ i_{E_0}$, where $k$ is the embedding given by the tail of the iteration and the critical point of $k$ is  some $M_{E_0}$-measurable cardinal greater than $\kappa$ (and so above  $(\kappa^{++})^{M_{E_0}}$). 
    Let $\gamma$ be the supremum of the generators\footnote{A generator of $E_0$ is an ordinal $\delta$ such that for every $\alpha<\delta$ and every $f:\kappa\to\kappa$, $j_U(f)(\alpha)\neq\delta$.} of $E_0$.
    Note that $\gamma\leq(\kappa^{++})^{M_{E_0}}$ since otherwise, by coherence and the initial segment condition on the extender sequence of the core model, $E_0\restriction (\kappa^{++})^{M_{E_0}}\in M_{E_0}$ and witnesses that \(\kappa\) is \(2\)-strong in \(M_{E_0}\), contradicting the anti-large cardinal assumption of the theorem.
    Also \((\kappa^{++})^{M_{E_0}} < (\kappa^{++})^K\), since otherwise \(E_0\) witnesses that \(\kappa\) is a \(2\)-strong cardinal in \(K\). 
    
    For each $\alpha<\gamma$, the measure $E_0(\alpha)$ is a subset of $U_\alpha$, where $U_\alpha$ is the $V$-ultrafilter derived from $j_U$ and $k(\alpha)$. In particular, $U_\alpha\leq_{RK} U$ via some function $f_\alpha:\kappa\to\kappa$.
    Since $2^\kappa=\kappa^+$ in ${K}$ and since $U$ is a $P_{\kappa^{++}}$-point, there is a set $B_\alpha\in U$ such that $f_\alpha[B_\alpha]\subseteq^* X$ for all $X\in E_0(\alpha)$: 
    let \(B_\alpha\in U\) be a \(\subseteq^*\)-lower bound of \(\{f_\alpha^{-1}[X] : X\in E_0(\alpha)\}\).
    Since $\gamma<\kappa^{++}$ and again since $U$ is a $P_{\kappa^{++}}$-point, we can find a single $B\in U$ such that $B\subseteq^* B_\alpha$ for all $\alpha<\gamma$. Note also that $E_0(\alpha)=\{X\in P^{K}(\kappa)\mid f_\alpha[B_\alpha]\subseteq^* X\}$. Since $j_U\restriction K$ is an iteration of $K$ with critical point $\kappa$, $P^K(\kappa)=P^{(K)^{M_U}}(\kappa)$. Using the fact that $f_\alpha[B_\alpha]\in M_U$ we have that $E_0(\alpha)\in M_U$.   
    
    Let $U'$ be the filter on \(\kappa\) that is \(\subseteq^*\)-generated by $B$. Then $U'\in M_U$. Let us claim that $E_0$ can be reconstructed in $M_U$ from \(U'\), which will lead to a contradiction (since it will imply that $E_0\in M_{E_0}$).
    \begin{claim}
    For each $\alpha<\gamma$, $E_0\restriction \alpha\in M_U$    
    \end{claim}
    \begin{proof}
        As we already noticed, $E_0(\alpha)\in M_U$. By Lemma \ref{lma:factor_crit}, applied to $j=j_{E_0}$, we conclude that $E_0\restriction \alpha$ is the extender of length \(\alpha\) derived from $j_{E_0(\alpha)}\restriction P^K(\kappa)$, which belongs to \(M_U\). 
    \end{proof}
    \begin{claim}
        $E_0\in M_U$.
    \end{claim}
    \begin{proof}
    We will prove that there is a formula \(\varphi(x_0,x_1,x_2,x_3)\) in the language of set theory such that for any \(\alpha < \gamma\), $E_0\restriction \alpha$ the unique \(F\in M_U\) such that \(M_U\vDash \varphi(F,f_\kappa,U',\alpha)\). Then \(\{E_0\restriction \alpha : \alpha < \gamma\}\in M_U\), which proves the claim.
    
    To be precise, \(\varphi(F,f_\kappa,U',\alpha)\) states that
    $F$ is a $K$-extender of length $\alpha$, $(2^\kappa)^{+{K}_F}\geq\alpha$,
    and there is a family of functions \(\langle g_a: a\in [\alpha]^{<\omega}\rangle\) such that:
    \begin{enumerate}
        \item Each $F_a\subseteq (g_a)_*(U')$.
        \item For each $a\subseteq b$, $\pi_{a,b}\circ g_b=g_a \mod U'$, where $\pi_{a,b}$ is the usual map from $\kappa^{|b|}$ onto $\kappa^{|a|}$.
        \item  $g_\kappa=f_\kappa$.
    \end{enumerate}
    By condition $(1)$, $(g_{a})_*(U')\cap K=F_a$. Since $U'\subseteq U$, this ensures that the maps $k_a:{K}_{F_a}\to K^{M_U}$ defined by $k_a([h]_{F_a})=[h\circ g_a]_U$ are well-defined and $j_U\restriction {K}=k_a\circ j_{F_a}$. Condition $(2)$ ensures that whenever $a\subseteq b$, $k_a=k_b\circ k_{a,b}$, where $k_{a,b}:{K}_{F_a}\to {K}_{F_b}$ is the usual factor map defined by $k_{a,b}([h]_{F_a})=[h\circ \pi_{b,a}]_{F_b}$. Indeed,
    $$k_a([h]_{F_a})=[h\circ g_a]_U=[h\circ \pi_{a,b}\circ g_b]_U=k_b([h\circ \pi_{a,b}]_{F_b})=k_b(k_{a,b}([h]_{F_a}).$$
    By the universal property of direct limits, the extender embedding $j_F:{K}\to{K}_F$ factors into $j_U\restriction {K}$; i.e., $j_U\restriction {K}=k\circ j_F$, where $k$ is the direct limit embedding of the $k_a$'s.
    
    Clearly, $E_0\restriction\alpha$ satisfies the above. For uniqueness, if $F$ satisfies the above then by requirement (3) that $f_\kappa=g_\kappa$, we have $F(\kappa)=E_0(\kappa)$ and $k(\kappa)=\kappa$. We claim that the critical point of $k$ is at least  $(2^\kappa)^{+{K}_F}$. To see this we simply note that $$P(\kappa)\cap {K}_F=P(\kappa)\cap {K}=P(\kappa)\cap j_U({K})$$ 
    and since $\crit(k)>\kappa$, for every $X\subseteq \kappa$, $k(X)=X$. It follows that for every $Y\subseteq P(\kappa)$, $Y\in {K}_F$, $k(Y)=Y$. It follows that every ordinal $\beta<(2^\kappa)^{+{K}_F}$, $k(\beta)=\beta$. 
    
    Finally note that $(2^\kappa)^{+{K}_F}=\crit(k)\geq \alpha$. Hence for every $a\in [\alpha]^{<\omega}$, $F(a)$ is the ultrafilter derived from $j_U$ and $a$, so \(F(a) = E_0(a)\), and hence \(F = E_0\restriction \alpha\). 
    \end{proof}
    Working in $M_U$, we appeal to the maximality of \(K\) \cite[Thm. 8.6]{SteelBook}. Since \(E_0\in M_U\) and \(E_0\) coheres the extender sequence of \(K^{M_U}\), \(E_0\in K^{M_U}\). But \(E_0\) is the first extender applied in the normal iteration leading to \(K^{M_U}\), so this is a contradiction.
\end{proof}
\subsection{Preserving Measurability with Mathias Forcing}\label{Sec: Math}
\begin{theorem}\label{thm: continuity points}
    Suppose $\kappa$ is measurable, the core model ${K}$ exists, and $U\in {K}$ is a normal measure on $\kappa$. Assume that there is a pseudo-intersection \(A\) of \(U\) such that \(A\cap \text{Lim}(A)\) is unbounded. Then in ${K}$, $\kappa$ carries a $\mu$-measure.
\end{theorem}
\begin{proof}
Let $W\in V$ be a $\kappa$-complete ultrafilter over $\kappa$. 
\begin{claim}
    Let $\alpha\in j_W(A)\setminus\kappa$ and let $W_\alpha$ be the $V$-ultrafilter derived from $j_W$ and $\alpha$, then $W_\alpha\cap {K}=U$.
\end{claim}
\begin{proof}
    It suffices to prove that $U\subseteq W_\alpha$. For any $X\in U$, by assumption there is $\xi<\kappa$ such that $A\setminus \xi\subseteq X$. Hence $j_W(A)\setminus \xi\subseteq j_W(X)$. Since $\alpha\in j_W(A)\setminus \kappa$, it follows that $\alpha \in j_W(X)$ and thus $X\in W_\alpha$.
\end{proof}
By Schindler \cite{schindler_2006} (and Steel) again, $j_W\restriction {K}$ is an iteration of ${K}$ by its measures/extenders. Let \(\langle K_\alpha, E_\alpha, i_{\alpha,\beta} : \alpha < \beta \leq \theta\rangle\) be the normal iteration of ${K}$ such that $i_{0,\theta} = j_W\restriction{K}$; thus \(E_\alpha\) is the extender used at stage \(\alpha\) and \(i_{\alpha,\beta} : K_\alpha\to K_\beta\) is the canonical embedding. 
\begin{claim} Let $\alpha\in j_W(A)\setminus\kappa$.
    \begin{enumerate}
        \item  $\alpha$ is an image of $\kappa$ under the iteration. Namely, $\alpha=i_{0,\gamma}(\kappa)$ for some $\gamma<\theta$.
        \item Suppose that $\gamma'\geq \gamma$ is the first stage of the iteration where we apply an extender $E_{\gamma'}$ with critical point at least $\alpha$. Then $E_{\gamma'}(\alpha)=i_{0,\gamma'}(U)$.
    \end{enumerate}
\end{claim}
\begin{proof}
    For $(1)$, first note that $\alpha$ is a sky point; namely, that for every club $C\in{K}$ on $\kappa$, $\alpha\in j_W(C)$. This is true since $U$ is a normal measure. Now it is not hard to see that for any $\rho<\theta$, and every $i_{0,\rho}(\kappa)<\nu<i_{0,\rho+1}(\kappa)$, there is a function $f:\kappa\to\kappa$ in ${K}$ such that $\nu\leq i_{\rho+1}(f)(i_{0,\rho}(\kappa))$. Hence $\alpha$ must be of the form $i_{0,\gamma}(\kappa)$ for some $\gamma<\theta$.

    For $(2)$, we first note that $i_{0,\gamma'}[U]\cup F_\alpha\subseteq E_{\gamma'}(\alpha)$, where $F_\alpha$ is the tail filter on $\alpha$. To see this, let $X\in U$, then $\alpha\in j_W(X)$ hence $\alpha\in i_{\gamma'+1,\theta}(i_{\gamma',\gamma'+1}(i_{0,\gamma'}(X)))$. By the normality of the iteration, $\alpha\in i_{\gamma',\gamma'+1}(i_{0,\gamma'}(X))$ which implies that $i_{0,\gamma'}(X)\in E_{\gamma'}(\alpha)$. To see that $i_{0,\gamma'}(U)=E_{\gamma'}(\alpha)$ it suffices to prove that $i_{0,\gamma'}[U]\cup F_\alpha$ generates $i_{0,\gamma'}(U)$. This follows from the normality of $U$ and since every set in $i_{0,\gamma'}(U)$, is of the form $i_{0,\gamma'}(f)(\vec{\xi})$ for some $f:[\kappa]^{\vec{\xi}}\to U$, $f\in{K}$ and $\vec{\xi}\in [\alpha]^{<\omega}$. (See \cite[Lemma 3.11]{UltrafilterTheory}.)
\end{proof}
Now we are ready to prove that \(\kappa\) carries a  $\mu$-measure in ${K}$. Suppose not, towards a contradiction. Pick any point $\alpha^*\in j_W(A)\cap \text{Lim}(j_W(A))$ above $\kappa$. Then by the claim, fix \(\nu < \theta\) such that \(\alpha = i_{0,\nu}(\kappa)\) and stages $\{\nu_i\mid i\leq\eta\}$ of the iteration such that \(\nu = \nu_\eta = \sup_{i<\eta}\nu_i\), and for each $i\leq\eta$, at stage $\nu_i$ of the iteration, we apply an extender \(E_{\nu_i}\) whose derived normal measure is $i_{0,\nu_i}(U)$.

Since \(\kappa\) is not \(\mu\)-measurable in \(K\),
\(i_{0,\nu_i}(\kappa)\) is not \(\mu\)-measurable in \(K_{\nu_i}\), and so the extender \(E_{\nu_i}\) is actually equivalent to its derived normal measure \(i_{0,\nu_i}(U)\); otherwise, by the initial segment condition, the derived normal measure would belong to \(\Ult(K_{\nu_i},E_{\nu_i})\),
which implies \(E_{\nu_i}\) is a \(\mu\)-measure.

Now the measure  $i_{0,\nu_\eta}(U)$ is definable in \(M_W\) as the set of all $X\subseteq\alpha^*$ that contain a tail of $j_W(A)\cap \alpha^*$. Applying the maximality of the core model (for example, \cite[Theorem 8.14 (2)]{SteelBook}) in \(M_W\), $i_{0,\nu_{\eta}}(U)\in {K}^{M_W}$. Since the iteration is normal, we conclude that $i_{0,\nu_{\eta}}(U)\in i_{0,\nu_{\eta}+1}({K})$, which is itself the ultrapower of $i_{0,\nu_\eta}(K)$ by $i_{0,\nu_\eta}(U)$. Contradiction. 
\end{proof}
\begin{remark}
    Note that the assumption that $A$ contains unboundedly many closure points is essential. Indeed, after Radin forcing with a repeat point, $\kappa$ stays measurable and there is a ground model normal measure which is diagonalized by the successor points of the Radin club.
\end{remark}
Let us use Theorem \ref{thm: continuity points} to provide a lower bound on the preservation of measurability after  the generalized Mathias forcing. This is related to the attempt to obtain a small ultrafilter number at a measurable cardinal using this method.
\begin{definition}\label{Def:Mathias forcing}
    Suppose $\kappa^{<\kappa}=\kappa$. Given a $\kappa$-complete filter $F$ over a measurable cardinal $\kappa\geq\omega$, let $\mathbb{M}_F$ be the forcing notion whose conditions are pairs $(a,A)\in [\kappa]^{<\kappa}\times F$. The order is defined by $(a,A)\leq (b,B)$ if $b\subseteq a$,  $A\subseteq B$, and $a\setminus b\subseteq B$.
\end{definition}
This forcing is $\kappa$-closed and $\kappa$-centered. This is an unorthodox definition, but it is forcing equivalent to the standard one where in the definition of $(a,A)\leq (b,B)$  we replace $b\subseteq a$ with $b\sqsubseteq a$. Indeed, consider the set of conditions $\mathbb{M}^*_U=\{(a,A)\mid \min(A)>\sup(a)\}$. Clearly $\mathbb{M}^*_U$ is dense in $\mathbb{M}_U$ and if $(a,A)\leq (b,B)\in\mathbb{M}^*_U$, then $\min(a\setminus b)>\sup(b)$, hence $b\sqsubseteq a$. The reason for presenting the forcing this way is the following simple lemma:
\begin{lemma}\label{Lemma: Proj}
    If $U\leq_{RK}W$ then $\mathbb{M}_W$ projects onto $\mathbb{M}_U$.
\end{lemma}
\begin{proof}
    Let $f:\kappa\to\kappa$ witness that $U\leq_{RK} W$, we may assume that $f$ is onto. Define $\phi:\mathbb{M}_W\to\mathbb{M}_U$ by $\phi((a,A))=(f''a,f''A)$ and we claim that $\phi$ is a  projection. If $(a,A)\leq (b,B)$, then $f''b\subseteq f''a$ and $f''A\subseteq f'' B$. Also if $\nu\in f''a\setminus f''b$, the $\nu=f(x)$ for some $x\in a\setminus b\subseteq B$, hence $\nu=f(x)\in f''B$. So $(f''a,f''A)\leq (f''b,f''B)$. Suppose that $(x,X)\leq (f''a,f''A)$. This means that $x\setminus f''a\subseteq f''A$. Hence there is $a'\subseteq A$ such that $f''[a\cup a']=x$. Also since $X\in U$, $f^{-1}[X]\in W$. Consider the condition $p=(a\cup a', A\cap f^{-1}[X])$. Then $p\leq (a,A)$ and $\phi(p)\leq (x,X)$. Hence $\phi$ is a projection.
\end{proof}
\begin{proposition}\label{Prop: Properties of mathias generic set}
    Suppose $G\subseteq \mathbb M_U$ is $V$-generic and \[A_G=\bigcup\{a\mid \exists A, \ (a,A)\in G\}.\] 
    \begin{enumerate}
        \item For every $A\in U$, $A_G\subseteq^* A$.
        \item If $\Cub_\kappa\subseteq U$, then $A_G\cap Lim(A_G)$ is unbounded in $\kappa$.
    \end{enumerate}
\end{proposition}
\begin{proof}
    The first item is clear, since every condition $(x,X)$ can be extended to a condition $(x,X\cap A)$ which forces that $\dot{A}_G\setminus\check{x}\subseteq \check{A}$. For the second item, Let $(x,X)$ be an condition, and $\delta<\kappa$, we will find a stronger condition which forces some continuity into $A_G$. Consider $Lim(X)\in \Cub_\kappa$. Then $X\cap Lim(X)\setminus \sup(x)\in U$. Let $\alpha>\delta$ be any point in $X\cap Lim(X)$, then $(x\cup X\cap \alpha+1,X\setminus \alpha+1)$ forces that $\alpha$ is a continuity point of $\dot{A}_G$ above $\delta$.
\end{proof}
\begin{corollary}
    Suppose that $V[G]$ is a generic extension where $\kappa$ is measurable, and there is $A\in V[G]$, a $V$-generic set for $\mathbb{M}_U$, where $U$ is a $\kappa$-complete ultrafilter in $V$. Then there is an inner model with a $\mu$-measurable cardinal.
    \end{corollary}
\begin{proof}
    By Lemma \ref{Lemma: Proj}, we may assume that $A$ is $V$-generic for $\mathbb{M}_U$ for a normal ultrafilter $U$ in $V$. By Proposition \ref{Prop: Properties of mathias generic set}, $A$ diagonalizes the $K$-normal measure $U\cap K$ and has unboundedly many continuity points. Hence we may apply Theorem \ref{thm: continuity points}.
\end{proof}

\subsection{The Generalized Tower Number}\label{Section: Applications}
Our first application is to give a non-trivial lower bound on the statement ``$\kappa$ is measurable and $\mathfrak{t}_\kappa>\kappa^+$". 
\begin{definition}\label{def: tower and pseudo intersection}
    A family $\mathcal{A}\subseteq [\kappa]^\kappa$ has the $\kappa$-SIP (strong intersection property) if for every $\mathcal{B}\in [\mathcal{A}]^{<\kappa}$, $\bigcap\mathcal{B}$ has size $\kappa$. A pseudo-intersection for $\mathcal{A}$ is a set $X\in [\kappa]^\kappa$ such that for every $A\in\mathcal{A}$, $X\subseteq^* A$. A tower in $\kappa$ is a sequence $\mathcal A = \l A_i\mid i<\lambda\r\subseteq [\kappa]^\kappa$ such that if $i<j$ then $A_i\supseteq ^* A_j$ and $\mathcal A$ has no pseudo-intersection. The generalized pseudo-intersection and tower numbers are defined as follows:
    \begin{enumerate}
        \item $\mathfrak{p}_\kappa$ is the minimum cardinality of a set \(\mathcal  A\subseteq [\kappa]^{\kappa}\) that has the \(\kappa\)-SIP but has no pseudo-intersection.
        \item $\mathfrak{t}_\kappa$ is the minimum length of a tower in $\kappa$.
    \end{enumerate}
\end{definition}
It is known that $\kappa^+\leq \mathfrak{p}_\kappa\leq \mathfrak{t}_\kappa\leq \mathfrak{b}_\kappa$ (see \cite[Lemma 31]{BROOKETAYLOR201737}).
Note that starting with an indestructible supercompact cardinal \(\kappa\) and an appropriate bookkeeping argument, one can iterate Mathias forcing of length $\kappa^{++}$ with ${<}\kappa$-support to add a diagonalizing set to any $\kappa$-complete uniform filter on \(\kappa\) which is generated by $\kappa^+$-many sets. This forcing preserves the supercompactness of $\kappa$ and makes $\mathfrak{p}_\kappa=\mathfrak{t}_\kappa=\kappa^{++}$. 
In the other direction, if one wishes to obtain $\mathfrak{t}_\kappa\geq\kappa^{++}$ at a measurable cardinal \(\kappa\), one must violate GCH at a measurable, which already implies an inner model where $o(\kappa)=\kappa^{++}$. Let us improve this lower bound:  
\begin{theorem}
    Suppose $\kappa$ is measurable and that $\mathfrak{t}_\kappa>\kappa^+$. Then there is an inner model with a $\mu$-measurable cardinal. 
\end{theorem}
\begin{proof}
    We first sketch a proof that the existence of a \(\pi P_{\kappa^{++}}\)-point implies an inner model with a \(\mu\)-measurable.
    In Gitik's argument  to obtain a $\mu$-measurable from a $P_{\kappa^{++}}$-point $U$ (which appears in \cite{tomCohesive}), we needed to reconstruct $U\cap{K}$ in the ultrapower $M_U$, and this was done by finding a set $A\in U$ such that $A\subseteq^* X$ for all $X\in U\cap{K}$. The purpose of the set $A$ is to define a filter $F\in M_U$ which includes $U\cap K$. It follows that the assumption of $A$ being a member of $U$ can be replaced with $A$ being unbounded in $\kappa$. 
    Therefore the argument works assuming that $U$ is a $\pi P_{\kappa^{++}}$-point (Definition \ref{def:pi_plambda_point}). From this point on, 
    the argument is identical to Gitik's. 
    
    To conclude the theorem, we claim that if $\mathfrak{t}_\kappa>\kappa^+$ and $U$ is normal, then $U$ is a $\pi P_{\kappa^{++}}$-point. Otherwise, let $\l X_i\mid i<\kappa^+\r\subseteq U$ be a counterexample. Since $U$ is normal, we can find a $\subseteq^*$-decreasing sequence $\l Y_i\mid i<\kappa^+\r\subseteq U$ such that for each $i<\kappa^+$, $Y_i\subseteq^* X_i$. The sequence of $Y_i$'s has no pseudo-intersection, since any such pseudo-intersection would have also been one for the sequence $\l X_i\mid i<\kappa^{+}\r$. Hence we see that $\l Y_i\mid i<\kappa^+\r$ is a tower, contradicting $\mathfrak{t}_\kappa>\kappa^+$. 
\end{proof}
This is related to a question of Gitik and Ben-Neria \cite[Question 3.2]{OmerMoti} which asked a similar question regarding the splitting number.

\subsection{Filter games without GCH}\label{Sec: Filter}
The filter games of Holy-Schlicht, Nielsen-Welch and Foreman-Magidor-Zeman revolve around several filter games defined as follows:

Fix $\theta$ a regular large enough cardinal. A transitive set $M$ is called a \textit{$\kappa$-suitable model} if  $M\subseteq H(\kappa^+)$ satisfies $ZFC^-$ and is  closed under ${<}\kappa$-sequences.

The notion of a constraint function defined below is essentially a notational tool to allow us to define several families of filter games all at once. 
\begin{definition}
    A \textit{constraint function} is a function \(\mathcal C\) that assigns to each \(\kappa\)-suitable model \(M\) a set \(\mathcal C(M)\) of \(\kappa\)-complete uniform filters on \(\kappa\) such that for each \(F\in \mathcal C(M)\), \(F\cap M\) is an \(M\)-ultrafilter.
\end{definition}

We will consider the following constraint functions:
\begin{enumerate}
    \item $\text{Set}(M)$ is the collection of all filters $F$ such that $F\cap M$ is a $\kappa$-complete $M$-ultrafilter and $F$ is $\subseteq^*$-generated by a single set.
    \item $\text{NSet}(M)$ is the collection of all filters $F$ such that $F\cap M$ is an $M$-normal ultrafilter and $F$ is $\subseteq^*$-generated by a single set.
    \item $\text{Filter}(M)$ is the collection of all filters $F$ such that $F\cap M$ is a $\kappa$-complete $M$-ultrafilter.
\end{enumerate}

\begin{definition}[The filter game]
Let $\kappa$ be a regular cardinal and let $\mathcal{C}$ be a constraint function. The filter game $G_{\mathcal{C}}(\kappa,\gamma)$ is the two-player game of length \(\gamma\) defined as follows: 

At stage $i$ of the game, Player I plays first a $\kappa$-suitable model $M_i$ of size at most $\kappa\cdot |i|$, such that $\bigcup_{j<i}M_j\subseteq M_i$. Then Player II responds with a filter $F_i\in\mathcal{C}(M_i)$ which extends $\bigcup_{j<i}F_j$.

 The game is played for every stage $i<\gamma$. Player I wins if and only if at some stage  $i<\gamma$, Player II has no legal move.
\end{definition}

Recall the following observation of Holy-Schlicht~\cite[Observation 3.5]{HolySchlicht:HierarchyRamseyLikeCardinals}:
\begin{proposition}\label{prop: equivalentGCH} Suppose that $2^\kappa=\kappa^+$. 
    The following are equivalent:
    \begin{enumerate}
        \item Player $II$ has a winning strategy in the game $G_{\text{NSet}}(\kappa,\kappa^+)$.
        \item Player $II$ has a winning strategy in the game $G_{\text{Set}}(\kappa,\kappa^+)$.
        \item Player $II$ has a winning strategy in the game $G_\text{Filter}(\kappa,\kappa^+)$.
        \item $\kappa$ is measurable.
    \end{enumerate} 
\end{proposition}
This proposition shows that assuming GCH, the filter games of length \(\kappa^+\) associated to any of the various constraint functions above are equivalent. If $2^\kappa>\kappa^+$, this is no longer obvious, and moreover, it makes sense to consider $G_{\mathcal C}(\kappa,\gamma)$ for \(\gamma > \kappa^+\).

We first show that the games of length \(\kappa^+\) are still equivalent in this context:
\begin{proposition}
    The following are equivalent:
    \begin{enumerate}
        \item Player $II$ has a winning strategy in the game $G_{\text{NSet}}(\kappa,\kappa^+)$.
        \item Player $II$ has a winning strategy in the game $G_{\text{Set}}(\kappa,\kappa^+)$.
        \item Player $II$ has a winning strategy in the game $G_\text{Filter}(\kappa,\kappa^+)$.
        \item $\kappa$ is measurable in \(V[G]\) where \(G\subseteq \text{Add}(\kappa^+,1)\) is \(V\)-generic.
    \end{enumerate} 
\end{proposition}
\begin{proof}
    (1) implies (2)  and (2) implies (3) are trivial. So let us begin by showing that (3) implies (4). We note that in $V[G]$, we have that $2^\kappa=\kappa^+$ regardless of the cardinal arithmetic of the ground model.
    By the \(\kappa^+\)-closure of the forcing, every winning strategy for Player II in the game \(G_\text{Filter}(\kappa,\kappa^+)\) in \(V\) remains a winning strategy in \(V[G]\). Therefore by Proposition \ref{prop: equivalentGCH}, \(\kappa\) is measurable in \(V[G]\).

    Finally, we show that (4) implies (1). Suppose that in $V[G]$, $\kappa$ is measurable and let $U$ be a normal ultrafilter on $\kappa$. Let \(\dot{U}\) be a name such that \(\dot{U}_G = U\). Consider
    the strategy for Player II in \(G_\text{NSet}(\kappa,\kappa^+)\)
    defined as follows. At stage \(i < \kappa^+\), 
    we will have defined a decreasing sequence
    \((p_j)_{j < i}\subseteq \text{Add}(\kappa^+,1)\). We choose a lower bound \(p_i\) of these conditions, forcing \(\dot{U}\cap M_i = \check D\), and then Player II plays the filter \(U_i\) that is $\subseteq^*$-generated by the diagonal intersection of \(D\).
\end{proof}

\begin{definition}
    A \(\kappa\)-suitable model \( M\subseteq H(\kappa^+)\) is \textit{internally approachable} by a sequence  \(\langle N_\alpha : \alpha < \kappa^+\rangle\) of \(\kappa\)-suitable models if \( M = \bigcup_{\alpha < \kappa^+} N_\alpha\) and for all \(\beta < \kappa^+\),
    \(\langle N_\alpha : \alpha < \beta\rangle\in N_{\beta}\).
\end{definition}

\begin{definition}
    If \(M\) is a transitive set and \(X\in M\) is a set, an \(M\)-ultrafilter \(U\) on \(X\) is \textit{\(\kappa\)-amenable} if for any \(\mathcal A \subseteq P^M(X)\) with \(\mathcal A\in M\) and \(|\mathcal A|^M \leq \kappa\), 
    \(U\cap \mathcal A \in M\).
\end{definition}

\begin{theorem}
 Player I does not have a winning strategy in the game $G_\text{Filter}(\kappa,\kappa^+)$ if and only if there are stationarily many internally approachable, \(\kappa\)-suitable models ${M}\subseteq H(\kappa^+)$ such that there is a $\kappa$-amenable, $\kappa$-complete ${M}$-ultrafilter on \(\kappa\).
\end{theorem}
\begin{proof}
    Suppose Player I has a winning strategy $\tau$ for $G_{Filter}(\kappa,\kappa^+)$. Then there are club many ${M}\preceq (H(\kappa^+),\tau)$. We claim that for any such ${M}$, if ${M}$ is internally approachable by a sequence $\l N_\alpha\mid\alpha<\kappa^+\r$, then there is no $\kappa$-complete, $\kappa$-amenable ${M}$-ultrafilter. Otherwise, let $U$ be such an ${M}$-ultrafilter, and we will use \(U\) to produce a run \(r\) which is played according to \(\tau\) but is a win for Player II. (This just means that the run \(r\) has length \(\kappa^+\) and Player II follows the rules of the game.)
    
    At move \(\alpha < \kappa^+\), 
    let Player I play \(N = \tau(r\restriction \alpha)\), and let Player II respond
    with $U\cap N$. 
    In order for this to be a valid move for II, \(U\cap N\) has to measure all sets in \(N\), and for this, it is essential that \(N\subseteq M\) (since \(U\) is just an \(M\)-ultrafilter). 
    In fact, we will show that the model \(N\) is an element of \(M\). We do this by proving by induction that each proper initial segment of the run \(r\) is an element of ${M}$. Since \(M\preceq (H(\kappa^+),\tau)\), it will follow that \(N = \tau(r\restriction \alpha) \in M\).
    
    Suppose that $\alpha<\kappa^+$ and suppose that $r\restriction\beta\in{M}$ for all $\beta<\alpha$. Let $\gamma<\kappa^+$ be large enough so that $r\restriction\beta\in N_\gamma$ for all $\beta<\alpha$. Now $r\restriction\alpha$ is definable in $(H(\kappa^+),\tau)$ from the parameter $U\cap N_\gamma$. Since $U\cap N_\gamma$ is a member of ${M}$ by $\kappa$-amenability and since ${M}$ is elementary in $(H(\kappa^+),\tau)$, $r\restriction\alpha\in {M}$. 

    In the other direction, suppose that Player $I$ does not have a winning strategy, and let $F:[H(\kappa^+)]^{<\omega}\to H(\kappa^+)$ be any function.  We will find an internally approachable model ${M}\subseteq H(\kappa^+)$ that is closed under $F$ and a $\kappa$-amenable, $\kappa$-complete ${M}$-ultrafilter. To do this, we will define a strategy for Player I, and then obtain a losing run played according to this strategy that will produce the desired ${M}$. 
    
    Let $r$ be a run in the game of length $\alpha<\kappa^+$, and we will define $\sigma(r)$. Assume by recursion have already defined $\sigma(r\restriction\beta)$ for every $\beta<\alpha$. Let $\sigma(r)$ be an elementary submodel of $(H(\kappa^+),F)$ of size $\kappa$ such that \[\{r,\l \sigma(r\restriction\beta)\mid \beta<\alpha\r\}\cup\alpha\subseteq \sigma(r)\] 

By our assumption, there is a winning run $r$ for Player II in which Player I plays according to $\sigma$.
Let $N_\alpha=\sigma(r\restriction\alpha)$ and ${M}=\bigcup_{\alpha<\kappa^+}N_\alpha$. It is clear that the union of the ultrafilters played by Player II is a $\kappa$-amenable, $\kappa$-complete ${M}$-ultrafilter.
\end{proof}

\begin{question}
\cite[Observation 3.5]{HolySchlicht:HierarchyRamseyLikeCardinals} shows that if \(\kappa\) is inaccessible and \(2^\kappa = \kappa^+\), then \(G_{Filter}(\kappa,\kappa^+)\) is determined. If \(2^\kappa > \kappa^+\), can the game fail to be determined?

Equivalently, suppose that there are stationarily many internally approachable, \(\kappa\)-suitable models \(M\subseteq H(\kappa^+)\) such that there is a \(\kappa\)-amenable, \(\kappa\)-complete \(M\)-ultrafilter on \(\kappa\). Must \(\kappa\) be measurable in \(V^{\text{Add}(\kappa^+,1)}\)?
\end{question}

Moving past $\kappa^+$, we first note that:
\begin{remark}
    The following are equiconsistent:
    \begin{enumerate}
        \item $o(\kappa)=\kappa^{++}$
        \item \(2^\kappa > \kappa^+\) and Player II has a winning strategy in the game $G_\text{Filter}(\kappa,2^\kappa)$.
\end{enumerate}
\begin{proof}
    The equivalence  follows from Gitik and Woodin's computation of the consistency strength of the failure of GCH at a measurable cardinal \cite{GitFailOfSch}, once we observe that (2) is equivalent to \(\kappa\) being measurable with \(2^\kappa > \kappa^+\). Indeed, if $\kappa$ is measurable, fix a $\kappa$-complete ultrafilter $U$ over $\kappa$. Then a winning strategy for Player $II$ is game $G_{Filter}(\kappa,2^\kappa)$ is given by responding $U\cap M$, whenever Player one plays a $\kappa$-suitable model $M$. In the other direction, let $\sigma$ be a winning strategy for Player II in the game $G_{Filter}(\kappa,2^\kappa)$. Let $\l M_\alpha\mid \alpha<2^\kappa\r$ be an increasing sequence of $\kappa$-suitable models such that $P(\kappa)\subseteq \bigcup_{\alpha<2^\kappa}M_\alpha$ such that $|M_\alpha|\leq|\alpha|$. Let $\l U_\alpha\mid \alpha<2^\kappa\r$ be the $M_\alpha$-ultrafilters produced by simulating a winning run using $\sigma$ and having Player $I$ playing the models $M_\alpha$. Let $U=\bigcup_{\alpha<2^\kappa}U_\alpha$. It is routine to verify that $U$ is a $\kappa$-complete ultrafilter over $\kappa$.
\end{proof}
\end{remark}
Given this remark and Theorem \ref{Thm: p-pointtostrong}, the following proposition shows that a winning strategy in the set game of length \(2^\kappa > \kappa^+\) has higher consistency strength than a winning strategy in the corresponding filter game, a distinction which does not arise when \(2^\kappa = \kappa^+\):
\begin{proposition} The following are equivalent for $\kappa$:
\begin{enumerate}
    \item Player II has a winning strategy in the game $G_\text{NSet}(\kappa, 2^\kappa)$.
    \item Player II has a winning strategy in the game $G_\text{Set}(\kappa, 2^\kappa)$.
    \item $\kappa$ carries a $P_{2^\kappa}$-point.
    
\end{enumerate}
\end{proposition}
\begin{proof}
    First we need the following claim:
    \begin{claim}
        If Player II has a winning strategy in the game $G_{Set}(\kappa,\lambda)$, then $\lambda\leq\mathfrak{b}_\kappa$. 
\end{claim}
\begin{proof}[Proof of Claim.] Let us show that any collection $\mathcal{C}$ of fewer than $\lambda$ clubs has a pseudo-intersection. Then we apply Proposition \ref{prop:club char of b,d} to conclude that $\lambda\leq \mathfrak{b}_\kappa$. Indeed, we enumerate $\mathcal{C}=\l C_\alpha\mid \alpha<\rho\r$ for some $\rho<\lambda$. Let $\sigma$ be a winning strategy in $G_{Set}(\kappa,\lambda)$. Consider a run of the game where Player II plays via $\sigma$ and Player I  plays at stage $\alpha<\rho$ a legal $\kappa$-suitable model $M_\alpha$ such that $C_\alpha\in M_\alpha$. Since $\rho<\lambda$, the run reaches stage $\rho$, and we can make Player I play any $\kappa$-suitable model $M_\rho$ including $\bigcup_{\alpha<\rho}M_\alpha$.Let $X_\rho$ be the response of $\sigma$. Then $X_\rho$ generates a $\kappa$-complete ultrafilter $U$ on $M_\rho$. Although $U$ might not extend $\Cub_\kappa\cap M$, by Theorem \ref{Thm: moving to ext of Club} and the subsequent Remark \ref{remark: adapts to models}, there is a $\kappa$-complete $M_\rho$-ultrafilter such that $\Cub_\kappa\subseteq M_{\rho}$ and $W\leq_{RB}U$. The Rudin-Blass projection will project $X_\rho$ to a set $Y$ which is a $\subseteq^*$-lower bound for $W$. Since $\mathcal{C}\subseteq M_\rho$, $\overline{Y}$ (the closure of $Y$) is a club which is a $\subseteq^*$-lower bound for $\mathcal{C}$, as desired.
\end{proof}
The key consequence of the claim is that (2) implies that \(2^\kappa\) is regular. This is because (2) implies \(2^\kappa = \mathfrak{b}_\kappa\), and \(\mathfrak b_\kappa\) is regular.

We now turn to the proof of the equivalence of (1), (2), and (3). The fact $(1)$ implies $(2)$ is trivial. To show that $(2)$ implies $(3)$, let $\sigma$ be a winning strategy for Player II in the game $G_{Set}(\kappa,2^\kappa)$. Let $\l A_\alpha\mid \alpha<2^\kappa\r$ be an enumeration of $P(\kappa)$. Consider the run of $G_{Set}(\kappa,2^\kappa)$ in which Player II plays by \(\sigma\) and at stage $\alpha < 2^\kappa$, Player $I$ plays a legal $\kappa$-suitable model $M_\alpha$ that contains $A_\alpha$ and all the sets $\l X_\beta\mid \beta<\alpha\r$ produced by $\sigma$ in the previous stages of the run. Let $U$ be the filter generated by  $\l X_\alpha\mid \alpha<2^\kappa\r$. Then \(U\) is indeed an $P_{2^\kappa}$-point ultrafilter since $U$ measures every set in $\bigcup_{\alpha<2^\kappa} M_\alpha$, and if $\mathcal{A}\subseteq U$ is of size less than $2^\kappa$, then since $2^\kappa$ is regular, and there is $\alpha<2^\kappa$ such that $\mathcal{A}\subseteq M_\alpha$. By the definition of the game $G_{Set}(\kappa,2^\kappa)$, the set $X_{\alpha+1}$ is a pseudo intersection of $U\cap M_\alpha$ and therefore of $\mathcal{A}$.
    Finally, to see that $(3)$ implies $(1)$, we fix a $P_{2^\kappa}$-point $U$.  
    The winning strategy for Player II is defined at stage $\alpha$ to return a set $X_\alpha$ that diagonalizes $U\cap M_\alpha$. Such a set $X_\alpha$ exists since $U$ is a $P_{2^\kappa}$-point and since $|M_\alpha|<2^\kappa$.
\end{proof}

In fact, a weaker hypothesis than a winning strategy in \(G_{Set}(\kappa,2^\kappa)\) already has consistency strength beyond \(o(\kappa) = \kappa^{++}\):
\begin{theorem}
    If Player II has a winning strategy in $G_{NSet}(\kappa,\kappa^++1)$, then there is an inner model with a $\mu$-measurable cardinal. 
\end{theorem}
\begin{proof}
    We may assume there is no inner model with a strong cardinal, and let \(K\) be the core model. We will repeatedly use the fact that if \(U\) is a \(K\)-ultrafilter and the ultrapower of \(K\)
    by \(U\) is well-founded, then \(U\in K\). This follows from \cite[Theorem 8.13]{SteelBook}.
    
    Let \(\sigma\) be a strategy for Player II in $G_{NSet}(\kappa,\kappa^++1)$. 
    Let \(\langle S_\alpha : \alpha < \kappa^+\rangle\) enumerate \(H(\kappa^+)\cap K\).
    We construct a run \(r\) of $G_{NSet}(\kappa,\kappa^+)$ in which Player I plays \(\kappa\)-suitable models \(M_\alpha\) with \(S_\alpha\in M_\alpha\) and Player II respondes according to \(\sigma\).
    Let \(A = \sigma(r^{\smallfrown}\l H(\kappa^+)\cap K\r)\) and let \(U\) be the \(K\)-normal \(K\)-ultrafilter
    \(\subseteq^*\)-generated by \(A\). Note that \(U\in K\) since \(U\) is a \(V\)-countably complete \(K\)-ultrafilter.

    Let \(T\subseteq \kappa^+\) be such that \(H(\kappa^+)\cap K\in L[T]\). Let \(N = L[A,T]\). Finally, let \(M = H(\kappa^+)\cap N\). Note that \(|H(\kappa^+)\cap N| =  \kappa^+\) since \((2^\kappa)^N \leq \kappa^+\) by a standard condensation argument. 

    We have that \(H(\kappa^+)\cap K = H(\kappa^+)\cap K^N\) since above \(\aleph_2\), \(K\) is obtained by stacking mice \cite[Lemma 3.5]{Gitik_Schindler_Shelah_2006}.
    
    Let $B = \sigma(r^\smallfrown\langle M\rangle)$, and let 
    \(W\) be the \(M\)-normal \(M\)-ultrafilter \(\subseteq^*\)-generated by \(B\).
    Let \(j : N \to N_W\) be the ultrapower of \(N\) by \(W\), which is well-founded since \(W\) is (truly) countably complete. 
    Since \(W\cap K = U\in K\),
    \(P(\kappa)\cap K = P(\kappa) \cap K^{N_W}\). 
    Since \(A \in N_W\) and \(P(\kappa)\cap K\in N_W\),
    \(U\in N_W\). Using the closure of \(K^{N_W}\) under \(N_W\)-countably complete ultrafilters, 
    \(U\in K^{N_W}\).

    Let \(D\) be the \(K\)-ultrafilter on \(V_\kappa\cap K\) derived from \(j\) and \(U\). 
    Then \(D\in K\) since (again) \(K\) is closed under countably complete ultrafilters. 
    Let \[k : (H(\kappa^+)\cap K)_D \to j(H(\kappa^+)\cap K)\] be the factor map.
    Note that \(k([\id]_D) = U\) and 
    \(k\restriction (\kappa+1)\) is the identity, so \([\id]_D = U\cap (H(\kappa^+)\cap K)_D =  U\).
    Therefore \(D\) witnesses that \(\kappa\) is \(\mu\)-measurable in \(K\).
\end{proof}
\begin{remark}
    A somewhat similar argument can be used to show that if Player II has a winning strategy in $G_{Set}(\kappa,\kappa^++2)$, then there is an inner model with a $\mu$-measurable cardinal. We leave open the question of whether a winning strategy for Player II in the game $G_{Set}(\kappa,\kappa^++1)$ already implies an inner model with a $\mu$-measurable cardinal.  
\end{remark}
\section{Questions}
\begin{question}
    What is the consistency strength of having $\mathfrak{t}_\kappa>\kappa^+$ for a regular cardinal $\kappa>\omega$?
\end{question}
Note that by Zapletal \cite{ZapletalSplitting}, this is at least $o(\kappa)=\kappa^{++}$. In this paper, we show that for a measurable cardinal $\kappa$, the consistency strength jumps above $o(\kappa)=\kappa^{++}$.
\begin{question}
    What is the exact consistency strength of 
    the existence of a $P_{\kappa^{++}}$-point on an uncountable cardinal \(\kappa\)? 
\end{question}

\begin{question}
    What is the consistency strength of Player II having a winning strategy in the game $G_{Filter}(\kappa,\kappa^++1)$?
\end{question}
Note that there is an upper bound of $o(\kappa)=\kappa^{++}$.
\bibliographystyle{amsplain}
\bibliography{ref}

\end{document}